\documentclass[a4paper,11pt]{amsart}
\usepackage[utf8]{inputenc}
\usepackage{amsmath}
\usepackage{amssymb}
\usepackage{amsfonts}
\usepackage{amsthm}
\usepackage{mathrsfs}
\usepackage{tikz-cd}
\usepackage{mathtools}
\usepackage{tensor}
\usepackage{geometry}
\usepackage{stmaryrd}
\usepackage{tikz}
\usepackage{pgfplots}\pgfplotsset{compat=1.13}
\usepackage{float}
\usepackage{framed, enumitem}
\usepackage{bm}
\usepackage{quiver}
\usepackage{hyperref}
\usepackage{todonotes}

\renewcommand{\sslash}{\mathbin{/\mkern-6mu/}}

\newcommand{\C}{\mathbb{C}}
\newcommand{\Z}{\mathbb{Z}}

\newcommand{\G}{\mathbb{G}}

\newcommand{\OO}{\mathcal{O}}
\newcommand{\PP}{\mathbb{P}}
\newcommand{\A}{\mathbb{A}}

\newcommand{\LL}{\mathcal{L}}

\newcommand{\Gr}{\mathrm{Gr}}

\renewcommand{\L}{\mathcal{L}}

\DeclareMathOperator{\codim}{codim}
\DeclareMathOperator{\spec}{Spec}

\DeclareMathOperator{\Lie}{Lie}
\DeclareMathOperator{\stab}{Stab}

\newtheorem{theorem}{Theorem}[section]
\newtheorem{lemma}[theorem]{Lemma}
\newtheorem{prop}[theorem]{Proposition}
\newtheorem{cor}[theorem]{Corollary}
\newtheorem{defn}[theorem]{Definition}

\theoremstyle{definition}
\newtheorem{construction}[theorem]{Construction}

\numberwithin{equation}{section}

\theoremstyle{remark}
\newtheorem*{remark}{Remark}

\title[Moduli Spaces of Hyperplanar Admissible Flags in Projective Space]{Moduli Spaces of Hyperplanar Admissible Flags in Projective Space}
\author{George Cooper}
\address{Centro di Ricerca Ennio De Giorgi, Collegio Puteano, Scuola Normale Superiore, Piazza dei Cavalieri, 3, I-56100 PISA}
\email{george.cooper@sns.it}

\begin{document}
	
\begin{abstract}
	We prove the existence of quasi-projective coarse moduli spaces parametrising certain complete flags of subschemes of a fixed projective space $\mathbb{P}(V)$ up to projective automorphisms. The flags of subschemes being parametrised are obtained by intersecting non-degenerate subvarieties of $\mathbb{P}(V)$ of dimension $n$ by flags of linear subspaces of $\mathbb{P}(V)$ of length $n$, with each positive dimension component of the flags being required to be non-singular and non-degenerate, and with the dimension $0$ components being required to satisfy a Chow stability condition. These moduli spaces are constructed using non-reductive Geometric Invariant Theory for actions of groups whose unipotent radical is graded, making use of a non-reductive analogue of quotienting-in-stages developed by Hoskins and Jackson.
\end{abstract}	

\maketitle

\setcounter{tocdepth}{1}
\tableofcontents

\section{Introduction}

In this paper we construct using non-reductive GIT quasi-projective coarse moduli spaces which parametrise certain classes of varying flags of subschemes $X^0 \subset X^1 \subset \cdots \subset X^n$ of a fixed projective space $\PP(V)$ up to the action of the group $PGL(V)$ of projective automorphisms.

The problem of moduli has long played an important role in Algebraic Geometry, and a significant portion of current research is dedicated to the study of moduli spaces of projective schemes and their descendents. For the past several decades, the main approach to constructing scheme-theoretic moduli spaces has been to use Mumford's (reductive) Geometric Invariant Theory (GIT), as developed in \cite{git}. Classical examples of GIT moduli spaces of projective schemes include non-singular hypersurfaces in projective space \cite{mfk}, stable curves \cite{gieseker} \cite{mumfordstability}, canonically polarised surfaces of general type \cite{giesekersurfaces} and canonically polarised varieties with canonical singularities of any dimension \cite{viehweg}. The first three of these constructions involve using the Hilbert--Mumford criterion to establish GIT stability. However, when working with higher dimensional schemes other approaches to constructing moduli spaces, involving GIT or otherwise, are often necessary; amongst other issues, it is not known how to directly implement the Hilbert--Mumford criterion to test for the GIT stability of subschemes $X$ of projective space with $\dim X \geq 3$, even when $X$ is smooth (Viehweg's construction relies on establishing deep positivity results to directly produce invariant sections of the linearisations on the Hilbert scheme he considers).

Related to the problem of moduli of schemes is the moduli of \emph{flags} of schemes, such as pairs of schemes with a divisor. The main result of the paper is that \emph{non-reductive} GIT yields constructions of examples of non-empty quasi-projective coarse moduli spaces of flags of schemes, namely when the schemes are subschemes of a projective space $\PP(V)$ and the flag $\underline{X}: X^0 \subset X^1 \subset \cdots \subset X^n$ is obtained by intersecting an $n$-dimensional variety $X^n \subset \PP(V)$ with a flag of linear subspaces $\underline{Z}: Z^0 \subset \cdots \subset Z^n = \PP(V)$ with $Z^i \subset \PP(V)$, of codimension $n-i$; here the flags $\underline{X}$ and $\underline{Z}$ are allowed to vary in $\PP(V)$ and $V$ respectively.

\subsection*{Summary of Results}

The objects of interest to us are known as \emph{(complete) hyperplanar admissible flags of subschemes of $\PP(V)$},\footnote{The terminology here is inspired by that of \cite{lazarsfeldmustata}. The hyperplanar admissible flags of interest to us are admissible flags in the sense of Lazarsfeld--Mustaţă, with the exception that we allow multiple points in dimension $0$, not just one.} which are flags $X^0 \subset X^1 \subset \cdots \subset X^n$ where $X^n$ is an $n$-dimensional subscheme of $\PP(V)$, and where each $X^i$ is an $i$-dimensional subscheme obtained by intersecting $X^n$ with a linear subspace $\PP(Z^i) \subset \PP(V)$ of codimension $i$, with $Z^0 \subset Z^1 \subset \cdots \subset Z^n = V$. The \emph{degree} of this flag is given by the common degree of the subschemes $X^i \subset \PP(V)$, and the \emph{Hilbert type} is given by the Hilbert polynomials of the $X^i$. Such a flag is said to be \emph{non-degenerate} if each $X^i \subset \PP(Z^i)$ is non-degenerate (i.e. not contained in a hyperplane in $\PP(Z^i)$), non-singular if $X^0$ is a disjoint union of reduced points and if the remaining $X^i$ are smooth connected varieties, and \emph{stable} if $X^0$ is a Chow stable length $0$ subscheme of $\PP(Z^0)$ (for more details, see Section \ref{section: the moduli functor}). Given any non-singular non-degenerate subvariety $X^n \subset \PP(V)$ then, provided the degree of $X^n$ is sufficiently large, the intersection of $X^n$ with a generic choice of a flag of linear subspaces $\PP(Z^0) \subset \cdots \subset \PP(Z^{n}) = \PP(V)$ of the appropriate dimensions is a non-degenerate, non-singular and stable hyperplanar admissible flag (see the discussion following Definition \ref{defn: complete flag}).

Our main result is that there are quasi-projective coarse moduli spaces parametrising non-degenerate, non-singular and stable hyperplanar admissible flags of subschemes of $\PP(V)$:

\begin{theorem} \label{theorem: main theorem}
	Let $n, d$ be positive integers, and let $V$ be a finite dimensional complex vector space. Assume $n + 1 < \dim V$ and $d > \dim V - n$, with $d \not \in \left\{ \frac{\dim V - n - 1 + i}{n + 1 - i} : i = 1, \dots, n \right\}$. Let $\underline{\Phi} = (\Phi_0, \dots, \Phi_n)$ be a tuple of Hilbert polynomials of subschemes of $\PP(V)$, where $\Phi_i(t) = \frac{d}{i!} t^{i} + O(t^{i-1})$ for each $i$. Let $\mathcal{F}_{n,d,\underline{\Phi}}^{\PP(V)}$ be the moduli functor parametrising equivalence classes of families of non-degenerate, non-singular and stable complete hyperplanar admissible flags of subschemes of $\PP(V)$ of length $n$, degree $d$ and Hilbert type $\underline{\Phi}$ (cf. Definition \ref{defn: equivalence of families}).
	
	Then there exists a quasi-projective coarse moduli space $M_{n,d,\underline{\Phi}}^{\PP(V)}$ for the moduli functor $\mathcal{F}_{n,d,\underline{\Phi}}^{\PP(V)}$.
\end{theorem}

In Section \ref{section: extending the main result}, we briefly indicate how Theorem \ref{theorem: main theorem} can be extended in a couple of ways, including incorporating weightings to the points of $X^0$, or by instead considering non-degenerate and non-singular hyperplanar admissible flags of the form $X^k \subset \cdots \subset X^n$, where in addition $X^k \subset \PP(Z^k)$ is required to be GIT stable. 

\subsection*{Summary of the Proof of the Main Result}

The proof of Theorem \ref{theorem: main theorem} proceeds as follows. In order to have a notion of a family of hyperplanar admissible flags, as part of the objects of the moduli problem we record the flag of linear subspaces $\PP(Z^0) \subset \cdots \subset \PP(Z^n)$, before imposing the equivalence relation of two flags being projectively equivalent. As such, we consider the diagonal action of $SL(V)$ on the product
$$ \mathcal{S} := \prod_{i=0}^n \mathrm{Hilb}(\PP(V), \Phi_i) \times \prod_{j=0}^n \mathrm{Gr}_j(V), $$
where $\mathrm{Gr}_j(V)$ is the Grassmannian of codimension $n-j$ linear subspaces of $V$ (see Section \ref{section: local universal property}). There exists an $SL(V)$-invariant locally closed subscheme $\mathcal{S}' \subset \mathcal{S}$ whose closed points correspond to non-degenerate, non-singular and stable hyperplanar admissible flags of subschemes of $\PP(V)$ of length $n$, together with the data of the closed embeddings $X^i \subset \PP(V)$ and the subspaces $Z^i \subset V$. The moduli space $M_{n,d,\underline{\Phi}}^{\PP(V)}$ is given by the geometric quotient of $\mathcal{S}'$ by $SL(V)$. 

A priori this is a problem in reductive GIT. However, since each of the subvarieties $X^i \subset \PP(V)$ for $i < n$ are degenerate, the Hilbert points $[X^i \subset \PP(V)] \in \mathrm{Hilb}(\PP(V), \Phi_i)$ are unstable. The subvariety $X^n \subset \PP(V)$ is non-degenerate; however, beyond the case of curves of degree $d$ and genus $g$ in $\PP^{d-g}$ \cite{bfmv}, in higher dimensions there is no complete algebraic classification of when $X^n \subset \PP(V)$ has a stable or unstable Hilbert point with respect to either the Hilbert or Chow linearisations on $\mathrm{Hilb}(\PP(V), \Phi_n)$. Thirdly, the flag of subspaces $Z^0 \subset \cdots \subset Z^{n-1} \subset Z^n = V$ defines an $SL(V)$-unstable point of $\prod_{j=0}^n \mathrm{Gr}_j(V)$; if $L_{\mathrm{Gr}_i}$ is the ample generator of $\mathrm{Pic}(\mathrm{Gr}_i(V)) \cong \Z$ (for $i < n$), then for any choice of weights $w_0, \dots, w_{n-1} \in \Z^{>0}$ the flag of subspaces is GIT unstable for the diagonal $SL(V)$-action on $\prod_{j=0}^{n-1} \mathrm{Gr}_j(V)$ with respect to the ample linearisation $\boxtimes_{j=0}^{n-1} L_{\mathrm{Gr}_j}^{w_j}$ (as seen by applying \cite[Theorem 11.1]{dolgachevbook} with $W = \PP(Z^{n-1})$).

To rectify these issues, we instead pass to the case of looking at the action of the parabolic subgroup $P \subset SL(V)$ preserving a \emph{fixed} flag of subspaces $Z^0 \subset Z^1 \subset \cdots \subset Z^n = V$ (where $\codim Z^i = n-i$), yielding a locally closed subscheme $\mathcal{S}_0' \subset \mathcal{S}'$ with $\mathcal{S}' \cong SL(V) \times^P \mathcal{S}_0'$, and use NRGIT to form a geometric quotient of $\mathcal{S}_0'$ by $P$; this quotient by $P$ coincides with the geometric quotient of $\mathcal{S}'$ by $SL(V)$. Rather than applying the \emph{$\hat{U}$-Theorem} to form this quotient in one go, we instead adopt the non-reductive GIT quotienting-in-stages procedure of Hoskins--Jackson \cite{hoskinsjackson}. We apply quotienting-in-stages directly, by verifying directly that at each stage the $\hat{U}$-Theorem can be applied, since this is more straightforward than determining whether the quotienting-in-stages assumptions of \emph{loc. cit.} hold.

There are two key features of our quotienting-in-stages approach. The first is that, since all grading 1PS involved have only two distinct weights, it is relatively straightforward to check that the necessary weight and stabiliser conditions needed to apply the $\hat{U}$-Theorem at each stage hold, at least for the objects of interest. The key observation is that the flat limit of $X^k \subset \PP(Z^k)$ under any of these grading 1PS is given by $X^k$ itself, or by the join $J$ of $X^j$ with a complementary linear subspace to $\PP(Z^j) \subset \PP(Z^k)$, where $j < k$ (cf. Lemmas \ref{lemma: flat limits are joins} and \ref{lemma: limits of nice configurations}). The geometry of these joins is what allows us to verify that the desired conditions on unipotent stabilisers hold. Since we elect to work with the Chow linearisations on the Hilbert schemes $\mathrm{Hilb}(\PP(V), \Phi_i)$, the Hilbert--Mumford weight of $J$ with respect to any of the grading 1PS can be explicitly computed (cf. Lemma \ref{lemma: chow weights of joins}); this allows us to verify the necessary minimality conditions for the weights of the grading 1PS (cf. Lemma \ref{lemma: weights of limits}).

Secondly, since we are working with complete flags $X^0 \subset \cdots \subset X^n$, we only ever have to consider the (reductive) GIT stability of subschemes of $\PP(V)$ of dimension $0$, which is classically understood.

\begin{remark}
	If one imposes the additional constraint that the top-dimensional subvariety $X^n \subset \PP(V)$ is GIT stable (for either the Chow or Hilbert linearisations), the construction of the coarse moduli space of hyperplanar flags $X^0 \subset \cdots \subset X^n \subset \PP(V)$ can be carried out using reductive GIT, by making use of \ref{lemma: schmitt product stability} to ensure that a point of $\mathcal{S}$ is $SL(V)$-stable if $[X^n \subset \PP(V)] \in \mathrm{Hilb}(\PP(V), \Phi_n)$ is GIT stable. If $X^n \subset \PP(V)$ is a smooth hypersurface of degree $d > 2$, it is classically known that the corresponding point of $\PP(H^0(\PP(V), \OO_{\PP(V)}(d)))$ is GIT stable. Another plentiful source of GIT stable subvarieties arises from a result of Donaldson \cite{donaldson}, which states that if $(X,L)$ is a smooth polarised manifold with discrete automorphism group for which there exist constant scalar curvature metrics in $c_1(L)$, then $(X,L)$ is asymptotically Chow stable (i.e. $X \subset \PP(H^0(X,L^k)^{\vee})$ is Chow stable for all $k \gg 0$). By work of Yau \cite{yau} (and independently in the canonically polarised case by Aubin \cite{aubin}), this includes canonically polarised smooth varieties and arbitrarily-polarised smooth Calabi--Yau varieties. More generally, Viehweg \cite{viehweg} proves that any canonically polarised variety with canonical singularities is asymptotically stable (with respect to linearisations arising as variants of the Hodge bundle), extending prior results for curves \cite{mfk} \cite{mumfordstability} and surfaces \cite{giesekersurfaces}. Other examples of Chow stable subvarieties of projective space arise from work by Morrison \cite{morrisonsurfaces} and Miranda \cite{miranda1981moduli}.
\end{remark}

\subsection*{Relation to Previous Work}

As mentioned above, the results of this paper rely heavily on Hoskins--Jackson's non-reductive quotienting-in-stages construction \cite{hoskinsjackson}. This work was also motivated in part by work of Ross \cite{rossthesis} and Ross--Thomas \cite{rossthomas} on the study of the Hilbert--Mumford criterion for polarised projective varieties. 1PS with two distinct weights feature prominently in these latter pieces of work, and are used to define notions of slope stability for polarised varieties, in analogy with Gieseker stability for coherent sheaves on a projective scheme (in turn, Gieseker stability is governed by 1PS with two distinct weights, as is explained in \cite[Section 3.2]{rossthesis}). Ross and Ross--Thomas prove that K-stable polarised varieties are slope stable and that asymptotically Chow stable polarised varieties are Chow-slope stable, and use slope stability to give a geometric proof of the asymptotic Chow/K-stability of smooth curves.

The relationship between \cite{rossthesis} \cite{rossthomas} and this paper is as follows. If $X \subset \PP(V)$ is a non-degenerate subvariety and $\lambda$ a 1PS of $GL(V)$ with two distinct weights (for instance, the lift $\lambda^{[i]}$ of the grading 1PS $\lambda_i$ considered in non-reductive quotienting-in-stages), the action of $\lambda$ on $X \subset \PP(V)$ defines a test configuration $(\mathcal{X}, \mathcal{L} = \OO_{\PP(V) \times \A^1}(1)|_{\mathcal{X}})$ of $(X, \OO_{\PP(V)}(1)|_X)$ whose central fibre is given by the flat limit $X_0 \subset \PP(V)$ of $X$ under $\lambda$. As explained in \cite[Page 8]{rossthomas} as well as in \cite[Section 3 of the proof of Theorem 2.9]{mumfordstability}, up to rescaling the $\lambda(\G_m)$-action the test configuration $(\mathcal{X}, \mathcal{L})$ is equivariantly dominated by a test configuration with total space $\mathrm{Bl}_{\mathcal{I}}(X \times \A^1)$, where $\mathcal{I} \subset \OO_{X}[t]$ is an ideal of the form $\mathcal{I} = \mathcal{I}_Z + (t)$, with $Z \subset X$ given by the repulsive fixed-point locus for $\lambda$. The blowup $\mathrm{Bl}_{\mathcal{I}}(X \times \A^1)$ coincides with the total space for the deformation to the normal cone of $Z \subset X$. The notion of slope stability (resp. Chow-slope stability) of a polarised variety $(X', L')$ arises from K-stability (resp. Chow stability) restricted to test configurations arising from deformations to normal cones of subschemes $Z' \subset X'$; when $L'$ is very ample and has no higher cohomology, so that $X' \subset \PP(H^0(X', L')^{\vee})$, by \cite[Proposition 3.7]{rossthomas} the test configuration of $(X', L')$ given by the deformation to the normal cone of $Z' \subset X'$ determines a 1PS of $GL(H^0(X', L'))$ with two distinct weights. 

In favourable circumstances (cf. proof of Lemma \ref{lemma: flat limits are joins}), the geometry of the deformation to the normal cone can be used to fully determine the flat limit $X_0$, in terms of the join of the repulsive locus $Z$ with a complementary linear subspace. In turn, the simple form taken by the homogeneous ideal of such a join allows for the $\lambda(\G_m)$-weight of $H^0(X_0, \OO_{X_0}(k))$ (with $k \gg 0$) to be easily computed in terms of the degree and dimension of $Z$ (cf. proof of Lemma \ref{lemma: chow weights of joins}), with the computations being more straightforward than those of Section 4 of \emph{loc. cit.} (which apply for the deformation to the normal cone in general).

\subsection*{Motivation}

This work was motivated by an attempt to construct, using GIT, a non-empty coarse moduli space of stable maps to a varying target subvariety $X$, with $X$ allowed to vary inside a fixed projective space $\PP(V)$, without any further conditions on the invertible sheaf $\OO_X(1) = \OO_{\PP(V)}(1) |_X$ beyond its amplitude, and with the image curve in $X$ corresponding to the intersection of $X$ with a linear subspace $L$ of $\PP(V)$ of the appropriate codimension (chosen so that, at least if $X$ is smooth and non-degenerate, the intersection of $X$ with a generic such $L$ is a smooth curve). This goal in turn arose as part of the more general goal of constructing moduli spaces of varieties together with morphisms from curves, as part of a wider aim of seeking new, GIT based approaches to constructing and studying moduli spaces of projective varieties by introducing suitable decorations.

The attempted construction involved appending to the Baldwin--Swinarski \cite{baldwinswinarski} GIT construction of the coarse moduli space of stable maps $\overline{M}_{g,n}(\PP(V), \beta)$ an additional Hilbert scheme factor, parametrising the subscheme $X \subset \PP(V)$, followed by attempting to form a geometric GIT quotient for the diagonal action of $SL(V)$ on the enlarged parameter space. The hope was that the Baldwin--Swinarski construction could be leveraged to show that the desired moduli space existed, without having to resort to implementing (in full) the Hilbert--Mumford criterion for the $SL(V)$-action on $\mathrm{Hilb}(\PP(V))$. The existence of this desired moduli space is closely related to the existence of a coarse moduli space parametrising partial flags $X^1 \subset X^n$ of $\PP(V)$. In the case $n=2$ a coarse moduli space can be constructed provided $X^1 \subset \PP(Z^1)$ is assumed to be a non-degenerate curve of genus $g$ and degree $d > 2(2g-2)$ (see Section \ref{section: omitting points}). However, it can be shown that when $n \geq 3$, in the resulting NRGIT setup the very stable locus (cf. Section \ref{section: NRGIT}) for the parabolic group action of interest is empty, and so Theorem \ref{theorem: U-hat in our setting} is no-longer applicable.

\subsection*{Notation and Conventions}

We adopt the following notation and conventions:

\begin{itemize}
	\item Throughout this paper we work over the field $\C$ of complex numbers. The term \emph{point} always refers to $\C$-points.
	\item A variety is an integral separated scheme of finite type.
	\item The term \emph{coarse moduli space} is to be understood in terms of $\mathbf{Set}$-valued moduli functors, as opposed to algebraic stacks (cf. Definition \ref{defn: cms}).
	\item If $\mathcal{E}$ is a locally free $\OO_T$-module of finite rank, we denote $\PP_T(\mathcal{E}) := \mathbf{Proj}_T(\mathrm{Sym}^{\bullet}(\mathcal{E}^{\vee}))$. Given a set of points $S$ of a projective space $\PP(V)$, their linear span is denoted $\langle S \rangle$. If $X \subset \PP(V)$ is a closed subscheme of $\PP(V)$ with ideal sheaf $\mathcal{I}_X$, we define $\mathbb{I}(X) := \bigoplus_{d \geq 0} H^0(\PP(V), \mathcal{I}_X(d))$ (we also use the notation $\mathbb{I}_{\PP(V)}(X)$ if we wish to emphasise the projective space $\PP(V)$ containing $X$). $\PP T_p X$ denotes the projective tangent space to $X \subset \PP(V)$ at a point $p \in X$.
	\item A closed subscheme $X$ of $\PP(V)$ is said to be \emph{non-degenerate} if $X$ is not contained in any hyperplane in $\PP(V)$, in other words if $H^0(\PP(V), \mathcal{I}_X(1)) = 0$, and is said to be \emph{degenerate} otherwise.
	\item If $G$ is an algebraic group acting on a quasi-projective scheme $Y$ and if $H \subset G$ is a subgroup, $G \times^H Y$ denotes the Borel construction, the geometric quotient of $H$ acting on $G \times Y$ by $h \cdot (g,y) = (gh^{-1}, hy)$.\footnote{By \cite[Theorem 4.19]{popovvinberg} $G \times^H Y$ is a scheme.}
	\item We use the following sign convention regarding the Hilbert--Mumford criterion in reductive GIT. Let $X$ be a projective scheme acted on by a reductive group $G$, let $\L$ be an ample linearisation for this action, let $\lambda$ be a 1PS (one-parameter subgroup) of $G$ and let $x \in X$, with limit $x_0 = \lim_{t \to 0} \lambda(t) \cdot x$. The \emph{Hilbert--Mumford weight} $\mu^{\L}(x, \lambda)$ is the weight of the $\lambda(\G_m)$-action on the fibre $\L^{\vee}|_{x_0}$; $x$ is (semi)stable if $\mu^{\L}(x, \lambda) > (\geq) \ 0$ for all non-trivial 1PS $\lambda$ of $G$.
\end{itemize}

\subsection*{Acknowledgements}

The contents of this paper constitutes a chapter of the author's DPhil thesis at the University of Oxford. The author would like to thank his supervisors Frances Kirwan, Jason Lotay and Alexander Ritter for their continuing guidance and encouragement during this project. The author also wishes to thank his thesis examiners Ruadhai Dervan and Dominic Joyce for their corrections and comments on this work. The author would also like to thank Eloise Hamilton, Victoria Hoskins and Josh Jackson for several helpful conversations and suggestions.

The author was supported by a UK Research and Innovation (Engineering and Physical Sciences Research Council) Doctoral Scholarship (EP/V520202/1).

\section{Preliminaries} \label{section: preliminaries}

\subsection{Lemmas Concerning Geometric Quotients}

Here we give a few elementary results needed in this paper concerning geometric quotients.

\begin{lemma} \label{lemma: restricting geometric quotients reductive}
	Let $q : X \to X/G$ be a geometric quotient for the action of a reductive algebraic group $G$ on a scheme $X$ of finite type. Let $Z \subset X$ be a $G$-invariant locally closed subscheme of $X$. Then $q$ restricts to give a geometric $H$-quotient $q|_{Z} : Z \to q(Z) \subset X/G$, where $q(Z)$ is a locally closed subscheme of $X/G$.
\end{lemma}

\begin{proof}
	Let $\overline{Z}$ be the closure of $Z$ in $X$; as $Z$ is locally closed in $X$ then $Z$ is open in $\overline{Z}$. Since $q$ is a good quotient then $q(\overline{Z})$ is closed in $X/G$; moreover $q(\overline{Z})$ inherits a canonical scheme structure from $\overline{Z}$ (if $\overline{Z} \subset X$ is locally $\spec A/I \subset \spec A$ then $q(\overline{Z}) \subset X/G$ is locally $\spec (A/I)^G = \spec A^G/I^G \subset \spec A^G$). The restriction $q|_{\overline{Z}} : \overline{Z} \to q(\overline{Z})$ is then a geometric $G$-quotient. 
	
	As $q$ is surjective, there is an equality $q(Z) = q(\overline{Z}) \setminus q(\overline{Z} \setminus Z)$, and so $q(Z)$ is open in $q(\overline{Z})$. Since the property of a morphism being a geometric quotient is local on the base, $q$ restricts to a geometric quotient $Z \to q(Z)$.
\end{proof}

\begin{lemma} \label{lemma: restricting geometric quotients locally trivial}
	Let $q : X \to X/H$ be a Zariski-locally trivial quotient for the action of an algebraic group $H$ (not assumed reductive) on a scheme $X$ of finite type. Let $Z \subset X$ be an $H$-invariant locally closed subscheme of $X$. Then there exists a unique scheme structure on the set-theoretic image $q(Z)$ and a locally closed immersion $q(Z) \to X/H$ such that $Z \cong q(Z) \times_{X/H} X$ as schemes and such that $q|_Z : Z \to q(Z)$ is a locally trivial quotient.
\end{lemma}

\begin{proof}
	Once again consider the closure $\overline{Z}$. The image $q(\overline{Z})$ is closed in $X/H$; we endow $q(\overline{Z})$ with the scheme-theoretic image scheme structure. Over an affine open $\spec A \subset X/H$ where $q$ is trivial, we have $q^{-1}(\spec A) = \spec A \times H$, where $H$ acts trivially on the first factor and by multiplication on the second. Writing $q(\overline{Z}) \cap \spec A = \spec A/I$ for $I \subset A$ an ideal, as $\overline{Z}$ is $H$-invariant we have $\overline{Z} = q^{-1}(q(\overline{Z})) = \spec A/I \times H$, and $q|_{\overline{Z}}$ is the projection $\spec(A/I) \times H \to \spec(A/I)$. It follows that $q|_{\overline{Z}} : \overline{Z} \to q(\overline{Z})$ is a locally trivial $H$-quotient and that $\overline{Z} \cong q(\overline{Z}) \times_{X/H} X$ as schemes. By restricting to the open subscheme $Z \subset \overline{Z}$, the same statements hold for $Z$ in place of $\overline{Z}$, since $q(Z)$ is open in $q(\overline{Z})$.
	
	The uniqueness of the scheme structure follows from the fact that locally trivial quotients are categorical quotients and categorical quotients are unique up to unique isomorphism.
\end{proof}

\begin{lemma} \label{lemma: borel construction}
	Let $G$ be a linear algebraic group acting on quasi-projective schemes $X$ and $Y$. Let $H \subset G$ be the stabiliser of a point $y_0 \in Y$, and let $Z = X \times \{y_0\} \subset X \times Y$. Then the map $\sigma_Z : G \times Z \to GZ$, $(g,z) \mapsto gz$ descends to an isomorphism $G \times^H Z \cong GZ$.
\end{lemma}

\begin{proof}
	The orbit map $\sigma_{y_0} : G = G \times \{y_0\} \to Gy_0$ is faithfully flat (cf. \cite[Proposition 7.4]{milnegroups}) and hence is a principal $H$-bundle (with respect to the fppf topology). As the diagram
	\[\begin{tikzcd}[ampersand replacement=\&]
		{G \times Z} \&\& {G \times \{y_0\}} \\
		\\
		GZ \&\& {Gy_0}
		\arrow["{\mathrm{id}_G \times \mathrm{pr}_Y}", from=1-1, to=1-3]
		\arrow["{\mathrm{pr}_Y}"', from=3-1, to=3-3]
		\arrow["{\sigma_Z}"', from=1-1, to=3-1]
		\arrow["{\sigma_{y_0}}", from=1-3, to=3-3]
	\end{tikzcd}\]
is Cartesian, $\sigma_Z$ is also a principal $H$-bundle, and in particular a geometric $H$-quotient. The lemma then follows from the uniqueness of geometric quotients.
\end{proof}

\subsection{The Chow Linearisation on Hilbert Schemes} \label{section: chow linearisation}

Let $\mathrm{Hilb}(\PP(V), \Phi)$ be the Hilbert scheme parametrising subschemes of $\PP(V)$ with Hilbert polynomial $\Phi$, where
$$ \Phi(t) = \frac{d}{n!} t^n + O(t^{n-1}). $$
Let $\mathrm{Chow}_{n,d}(\PP(V))$ be the Chow scheme of Angéniol \cite{angeniol} parametrising families of cycles in $\PP(V)$ of dimension $n$ and degree $d$. The Chow scheme has the following properties (cf. \cite[Sections 6-7]{angeniol} and \cite[Paper IV, Sections 9, 16 and 17]{rydhthesis}):
\begin{enumerate}[label=(\roman*)]
	\item There is a natural proper $GL(V)$-equivariant morphism $\Psi_{\mathrm{HC}} : \mathrm{Hilb}(\PP(V), \Phi) \to \mathrm{Chow}_{n,d}(\PP(V))$, known as the \emph{Hilbert--Chow morphism}. The restriction of $\Psi_{\mathrm{HC}}$ to the open subscheme $\mathrm{Hilb}^{\mathrm{smth}}(\PP(V), \Phi)$ parametrising non-singular closed subschemes of $\PP(V)$ is an open immersion.
	\item $\Psi_{\mathrm{HC}}$ is a local immersion over the locus parametrising equidimensional and reduced subschemes.
	\item The underlying reduced scheme $\mathrm{Chow}_{n,d}(\PP(V))_{\mathrm{red}}$ is the Chow variety, as introduced by Chow and van der Waerden in \cite{chowvdwaerden}.
	\item Let $\mathrm{Div}_{n,d}(\PP(V))$ be the projective space of multidegree $d$ divisors in $\PP(V^{\vee})^{\times(n+1)}$. Then there is a natural $GL(V)$-equivariant morphism $\Xi_{\mathrm{Chow}} : \mathrm{Chow}_{n,d}(\PP(V)) \to \mathrm{Div}_{n,d}(\PP(V))$ whose restriction to $\mathrm{Chow}_{n,d}(\PP(V))_{\mathrm{red}}$ is a closed immersion; the composite $\Xi_{\mathrm{HC}} = \Xi_{\mathrm{Chow}} \circ \Psi_{\mathrm{HC}}$ assigns to a subvariety $X \subset \PP(V)$ its Chow form $\Xi_X$. In particular $\Xi_{\mathrm{Chow}}$ is a finite morphism, and hence $\mathcal{L}_{\mathrm{Chow}} = \mathcal{L}_{\mathrm{Chow}_{n,d}(\PP(V))} := \Xi_{\mathrm{Chow}}^{\ast} \OO_{\mathrm{Div}_{n,d}(\PP(V))}(1)$ is an ample linearisation for the $GL(V)$-action on the projective scheme $\mathrm{Chow}_{n,d}(\PP(V))$, known as the \emph{Chow linearisation}.
	\item As the Chow form $\Xi_X$ of a subvariety $X \subset \PP(V)$ uniquely determines $X$, all points of the fibres $\Xi_{\mathrm{HC}}^{-1}(\Xi_X)$ and $\Xi_{\mathrm{Chow}}^{-1}(\Xi_X)$ have underlying reduced scheme $X$.
\end{enumerate}

These properties are summarised in the following diagram:
\[\begin{tikzcd}[ampersand replacement=\&,column sep=scriptsize]
	{\mathrm{Hilb}(\PP(V), \Phi)} \&\& {\mathrm{Chow}_{n,d}(\PP(V))} \&\& {\mathrm{Div}_{n,d}(\PP(V))} \\
	\\
	{\mathrm{Hilb}^{\mathrm{smth}}(\PP(V), \Phi)} \&\& {\mathrm{Chow}_{n,d}(\PP(V))_{\mathrm{red}} = \mathrm{ChowVar}_{n,d}(\PP(V))}
	\arrow["\shortmid"{marking}, hook', from=3-3, to=1-3]
	\arrow["\shortmid"{marking}, hook, from=3-3, to=1-5]
	\arrow["{\Xi_{\mathrm{Chow}}}", from=1-3, to=1-5]
	\arrow["{\Psi_{\mathrm{HC}}}", from=1-1, to=1-3]
	\arrow["\circ"{description}, hook', from=3-1, to=1-1]
	\arrow["\circ"{description}, hook, from=3-1, to=1-3]
	\arrow["{\Xi_{\mathrm{HC}}}", curve={height=-24pt}, from=1-1, to=1-5]
\end{tikzcd}\]

Suppose $\lambda : \G_m \to SL(V)$ is a 1PS, and $X \subset \PP(V)$ is a subvariety fixed by the action of $\lambda$. Let $S(X) = \mathrm{Sym}^{\bullet}(V^{\vee})/\mathbb{I}(X) = \bigoplus_m S(X)_m$ be the homogeneous coordinate ring of $X$. Let $w([X]_m, \lambda)$ be the $\lambda$-weight of the vector space $S(X)_m$.\footnote{If $W = \bigoplus_{i \in \Z} W_i$ is the weight space decomposition of a finite dimensional $\G_m$-representation $W$, the weight of $W$ is given by $\sum_i i \dim W_i$.}

\begin{prop}[\cite{mumfordstability}, Proposition 2.11] \label{prop: chow weight general}
	For large $m$, $w([X]_m, \lambda)$ is represented by a polynomial of the form $\frac{-a_X}{(n+1)!} m^{n+1} + O(m^n)$. The Hilbert--Mumford weight of $\Psi_{\mathrm{HC}}(X)$ is given by 
	$$ \mu^{\mathcal{L}_{\mathrm{Chow}}}(\Psi_{\mathrm{HC}}(X), \lambda) = \mu^{\OO_{\mathrm{Div}_{n,d}(\PP(V))}(1)}(\Xi_{\mathrm{HC}}(X), \lambda) = a_X. $$
\end{prop}

In the case where $\mathrm{Hilb}(\PP(V), d)$ is the Hilbert scheme of length $d$ subschemes of $\PP(V)$, $\Xi_{\mathrm{HC}}$ factors through the natural morphism
$$ \Xi_{\mathrm{HC}} : \mathrm{Hilb}(\PP(V), d) \to \mathrm{Sym}^d(\PP(V)) = \prod_d \PP(V) / \mathfrak{S}_d, $$
and this restricts to an open immersion on the open subscheme parametrising $d$ disjoint unordered reduced points in $\PP(V)$.

\begin{prop} \label{prop: chow stability dim 0}
	Fix a splitting $V = W \oplus W'$, so that we may regard $SL(W)$ as a subgroup of $SL(V)$ by declaring that $SL(W)$ acts trivially on vectors in $W'$. Suppose $Y \in \mathrm{Hilb}(\PP(V), d)$ corresponds to a set of distinct unordered points $\Xi_{\mathrm{HC}}(Y) = \{p_1, \dots, p_d\}$, with each $p_i \in \PP(W)$. Then $\Psi_{\mathrm{HC}}(Y) \in \mathrm{Chow}_{0,d}(\PP(V))$ is $SL(W)$-(semi)stable with respect to the Chow linearisation if and only if for all proper linear subspaces $Z \subset \PP(W)$,
	\begin{equation} \label{eq: chow stability for points}
		\frac{\#(Y \cap Z)}{d} < (\leq) \ \frac{\dim Z + 1}{\dim W}.
	\end{equation}
\end{prop}

\begin{proof}
	There is an $SL(W)$-equivariant closed immersion of Chow schemes $$ \mathrm{Chow}_{0,d}(\PP(W)) \hookrightarrow \mathrm{Chow}_{0,d}(\PP(V)), $$ 
	under which the Chow linearisation on $\mathrm{Chow}_{0,d}(\PP(V))$ pulls back to the Chow linearisation on $\mathrm{Chow}_{0,d}(\PP(W))$. As such, we may work instead with the $SL(W)$-action on $\mathrm{Chow}_{0,d}(\PP(W))$. The result then follows from \cite[Proposition 7.27]{mukaimoduli}.
\end{proof}

In the case where $\lambda$ is a 1PS of $SL(V)$ with two distinct weights and where $Y \subset \PP(V)$ is a closed subvariety contained entirely within a single weight space, the Chow weight of $Y$ can be computed in a very straightforward manner.

\begin{lemma} \label{lemma: chow weight with only one weight}
	Suppose we are given a decomposition $V = U \oplus W$, and let $Y \subset \PP(W)$ be a closed subvariety of dimension $r$ and degree $d$. Let $\lambda : \G_m \to SL(V)$ be the 1PS defined by declaring that each $u \in U$ has weight $a$ and each $w \in W$ has weight $b$, where $a \dim U + b \dim W = 0$ and where neither $a$ nor $b$ are zero. Then the Hilbert--Mumford weight of the point $\Psi_{\mathrm{HC}}(Y)$ (where $Y$ is considered as a closed subvariety of $\PP(V)$) is equal to $b d (r + 1)$.
\end{lemma}

\begin{proof}
	Pick bases $x_1, \dots, x_n$ and $y_1, \dots, y_m$ for $U^{\vee}$ and $W^{\vee}$ respectively, and let $I_Y = \mathbb{I}_{\PP(W)}(Y) \subset \mathrm{Sym}^{\bullet} W^{\vee} = \C[y_1, \dots, y_m]$. The homogeneous ideal of $Y \subset \PP(V)$ is given by
	$$ \mathbb{I}_{\PP(V)}(Y) = (I_Y, x_1, \dots, x_n) \cdot R \subset R := \C[x_1, \dots, x_n, y_1, \dots, y_m], $$
	so for all $k$ sufficiently large we have
	$$ H^0(Y, \OO_{\PP(V)}(k)|_Y) = \frac{R_k}{\mathbb{I}_{\PP(V)}(Y)_k} = \frac{\C[y_1, \dots, y_m]_k}{(I_Y)_k}, \quad h^0(Y, \OO_{\PP(V)}(k)|_Y) = \frac{d}{r!} k^r + O(k^{r-1}). $$
	All elements of $\C[y_1, \dots, y_m]_k/(I_Y)_k$ have weight $-kb$, so for all $k$ sufficiently large
	$$ w([Y]_k, \lambda) = -k b h^0(Y, \OO_{\PP(V)}(k)|_Y) = \frac{-b d (r+1)}{(r+1)!}k^{r+1} + O(k^r). $$
	Applying Proposition \ref{prop: chow weight general} completes the proof of the lemma.
\end{proof}	

\subsection{Joins of Varieties with Linear Subspaces and Flat Limits}

We recall the definition of the join of a subvariety $Y \subset \PP(V)$ with a linear subspace $L \subset \PP(V)$.

\begin{defn}
	Let $Y \subset \PP(V)$ be a subvariety and let $L \subset \PP(V)$ be a linear subspace with the property that $Y \cap L = \emptyset$. The \emph{join} of $Y$ and $L$ is the subvariety of $\PP(V)$ obtained as the union of all lines joining points of $Y$ with points of $L$:
	$$ J(Y,L) = \bigcup_{\substack{y \in Y \\ q \in L}} \langle y, q \rangle. $$
\end{defn}

Suppose $Y$ is contained in a linear subspace $L' = \PP(W) \subset \PP(V)$ disjoint from $L$ with the property that $\PP(V) = \langle L, L' \rangle$. The following facts are all standard.
\begin{enumerate}[label=(\roman*)]
	\item There is an equality $\mathbb{I}_{\PP(V)}(J(Y,L)) = \mathbb{I}_{\PP(W)}(Y) \cdot S$ of ideals of $S = \mathrm{Sym}^{\bullet}(V^{\vee})$.
	\item For each $q \in L$, $\PP T_q J(Y,L) = \PP(V)$.
	\item Suppose $p \in J(Y,L) \setminus L$ lies on the line joining $y \in Y$ to $q \in L$. Then $\PP T_p J(Y,L) = \langle L, \PP T_y Y \rangle$. If $y \in Y$ is a non-singular point then $\dim \PP T_p J(Y,L) = \dim_y Y + \dim L + 1$.
\end{enumerate}

Let $X \subset \PP(V)$ be a closed subvariety, and suppose $V = U \oplus W$ for non-zero subspaces $U, W \subset V$. Assume that $X$ is not contained in all of $\PP(W)$. Define a 1PS $\lambda$ of $GL(V)$ by declaring that all vectors $u \in U$ have weight $a$ and all vectors $w \in W$ have weight $b$, where $a < b$. Let $Y = X \cap \PP(W)$, and let $X_0 = \lim_{t \to 0} \lambda(t) \cdot X$ be the flat limit of $X$ under $\lambda$.

\begin{lemma} \label{lemma: flat limits are joins}
	In the above situation, assume in addition that the rational map $\mathrm{pr}_{\PP(U)} : X \dashrightarrow \PP(U)$ is dominant, that $Y$ is reduced and that the closed immersion $Y \hookrightarrow X$ is regular. Then the flat limit $X_0$ is given by the join $J(Y, \PP(U))$.
\end{lemma}

In order to prove Lemma \ref{lemma: flat limits are joins}, we require the following lemma.

\begin{lemma} \label{lemma: flat limit of image over origin}
	Let $f : Y \to Z$ be a morphism between schemes which are flat over $\A^1$. Let $f_0 : Y_0 \to Z_0$ be the base change of $f$ along $\spec \C \stackrel{0}{\to} \A^1$. Then there is an equality of schemes $\mathrm{im}(f_0) = \mathrm{im}(f)_{0}$, where $\mathrm{im}$ denotes the scheme-theoretic image.
\end{lemma}

\begin{proof}
	We have an immersion of schemes $\mathrm{im}(f_0) \subset \mathrm{im}(f)_{0}$, which we claim is an isomorphism. Without loss of generality we may assume $Y = \spec A$ and $Z = \spec B$ are affine, with $f$ corresponding to a homomorphism of $\C[t]$-algebras $\psi : B \to A$. Since $\mathrm{im}(\psi) \subset A$ contains no $t$-torsion, we have $\mathrm{Tor}_1^{\C[t]}(\mathrm{im}(\psi), \C[t]/t) = 0$. As such, applying $- \otimes_{\C[t]} \C[t]/t$ to the short exact sequence $0 \to \ker \psi \to B \to \mathrm{im}(\psi) \to 0$ gives the short exact sequence
	$$ 0 \to \frac{\ker \psi}{t \ker \psi} \to \frac{B}{tB} \to \frac{\mathrm{im}(\psi)}{t \ \mathrm{im}(\psi)} \to 0. $$
	In turn, there is a monomorphism $\ker \psi/t \ker \psi \hookrightarrow \ker \overline{\psi} \subset B/tB$. This induces a surjection
	$$ \frac{B/tB}{\ker \overline{\psi}} \twoheadrightarrow \frac{B/tB}{\ker \psi / t \ker \psi} = \frac{\mathrm{im}(\psi)}{t \ \mathrm{im}(\psi)} = \frac{B/\ker \psi}{t(B/\ker \psi)}. $$
	It follows that there is a closed immersion $\mathrm{im}(f)_0 \hookrightarrow \mathrm{im}(f_0)$, which by construction is inverse to $\mathrm{im}(f_0) \subset \mathrm{im}(f)_{0}$.
\end{proof}

\begin{proof}[Proof of Lemma \ref{lemma: flat limits are joins}]
	Consider the rational map
	$$ \Lambda' : X \times \A^1 \dashrightarrow \PP(V) \times \A^1, \quad ([u+w], t) \mapsto ([u + t^{b-a}w], t), $$
	and let $\mathcal{I} \subset \OO_{X \times \A^1} = \OO_X[t]$ denote the ideal sheaf of the base locus of $\Lambda'$. Blowing up $\mathcal{I}$, we have a diagram
	\[\begin{tikzcd}[ampersand replacement=\&,cramped]
		{B = \mathrm{Bl}_{\mathcal{I}}(X \times \A^1)} \\
		\\
		{X \times \A^1} \&\& {\PP(V) \times \A^1}
		\arrow[from=1-1, to=3-1]
		\arrow["\Lambda", from=1-1, to=3-3]
		\arrow["{\Lambda'}"', dashed, from=3-1, to=3-3]
	\end{tikzcd}\]
	Since $X \times \A^1$ is flat over $\A^1$, the same is true of $B$ (cf. \cite[Appendix B6.7]{fultonintersection}). By \cite[Lemma 2.13]{mumfordstability}, the scheme-theoretic image $\mathrm{im}(\Lambda) \subset \PP(V) \times \A^1$ is flat over $\A^1$, and the flat limit $X_0$ is given by the fibre of $\mathrm{im}(\Lambda)$ over $0 \in \A^1$. By Lemma \ref{lemma: flat limit of image over origin}, this coincides with the scheme-theoretic image of $B_0$. 
	
	Without loss of generality, we may assume that $b - a = 1$ (rescaling the weights of $\lambda$ will not change the flat limit as a subscheme of $\PP(V)$, only the induced $\lambda(\G_m)$-action on it). We have an equality
	$$ \mathcal{I} = \mathcal{I}_Y + (t) $$
	of ideals of $\OO_X[t]$, where $\mathcal{I}_Y$ is the ideal sheaf of $Y \subset X$. As such, $B = \mathrm{Bl}_{\mathcal{I}}(X \times \A^1) = \mathrm{Bl}_{Y \times 0}(X \times \A^1)$ is the total space for the deformation to the normal cone of $Y \subset X$ (cf. \cite[Chapter 5]{fultonintersection}). The fibre of $B$ over $0 \in \A^1$ is given by
	$$ B_0 = \mathrm{Bl}_Y X \cup_{E_Y X} \PP_Y(N_{Y \subset X} \oplus \OO_Y), $$
	noting that since $Y \subset X$ is a regular immersion, the normal cone $C_{Y \subset X} = N_{Y \subset X}$ coincides with the normal bundle, and is a (geometric) vector bundle. Here $\mathrm{Bl}_Y X$ and $\PP_Y(N_{Y \subset X} \oplus \OO_Y)$ intersect along the exceptional divisor $E_Y X$ and the section at infinity respectively. Since $X$ and $Y$ are reduced, the scheme $B_0$ is also reduced, and so $X_0$ must also be reduced (as scheme-theoretic images of reduced schemes are reduced). In other words, the scheme structure on $X_0$ coincides with the reduced induced scheme structure on the underlying space $|X_0|$. 
	
	On the other hand, since $\Lambda_0$ is the identity on $Y \subset B_0$, since $\Lambda_0$ is equivariant with respect to the $\G_m$-actions on $X_0$ and $\PP_Y(N_{Y \subset X} \oplus \OO_Y)$, and since $\Lambda_0$ coincides with the dominant morphism $\mathrm{pr}_{\PP(U)}$ on $\mathrm{Bl}_Y X \setminus E_Y X = X \setminus Y$, we must have $|X_0| = |J(Y, \PP(U))|$. It follows that $X_0 = J(Y, \PP(U))$ as subschemes of $\PP(V)$, completing the proof of the lemma.
\end{proof} 

\begin{lemma} \label{lemma: chow weights of joins}
	Suppose $V = U \oplus W$, and let $Y \subset \PP(W)$ be a closed subvariety of dimension $r$ and degree $d$. Let $J = J(Y, \PP(U))$ be the join of $Y$ and $\PP(U)$ in $\PP(V)$.
	
	Let $\lambda : \G_m \to SL(V)$ be the 1PS defined by declaring that each $u \in U$ has weight $a$ and each $w \in W$ has weight $b$, where $a \dim U + b \dim W = 0$ and where neither $a$ nor $b$ are zero. If $\dim Y = \dim \PP(U)$, assume further that $a + b d \neq 0$. Then the Hilbert--Mumford weight of the point $\Psi_{\mathrm{HC}}(J)$ with respect to $\lambda$ is equal to
	$$ \begin{cases} a (\dim \PP(U) + 1) & \text{if } \dim Y < \dim \PP(U), \\ b d (\dim Y + 1) & \text{if } \dim Y > \dim \PP(U), \\ (a + b d) (\dim Y + 1) & \text{if } \dim Y = \dim \PP(U). \end{cases} $$
\end{lemma}

\begin{proof}
	It is possible to use \cite[Theorem 4.8]{rossthomas} to prove this result, however we elect give a more direct, albeit similar calculation of the Hilbert--Mumford weight. As in the proof of Lemma \ref{lemma: chow weight with only one weight}, pick bases $x_1, \dots, x_n$ and $y_1, \dots, y_m$ for $U^{\vee}$ and $W^{\vee}$ respectively. Let $I_Y = \mathbb{I}_{\PP(W)}(Y) \subset \mathrm{Sym}^{\bullet} W^{\vee} = \C[y_1, \dots, y_m]$, and set
	$$ I_J := \mathbb{I}_{\PP(V)}(J) = I_Y \cdot R \subset R := \C[x_1, \dots, x_n, y_1, \dots, y_m]. $$
	For all $k$ sufficiently large, we have
	\begin{equation} \label{eq: weight space decomposition}
		H^0(J, \OO_{\PP(V)}(k)|_J) = \frac{R_k}{(I_J)_k} \cong \bigoplus_{i+j=k} \C[x_1, \dots, x_n]_i \otimes_{\C} \frac{\C[y_1, \dots, y_m]_j}{(I_Y)_j}
	\end{equation}
	Elements of $\C[x_1, \dots, x_n]_i$ all have weight $-i a$ and elements of $\C[y_1, \dots, y_m]_j/(I_Y)_j$ all have weight $-j b$. As such, by Equation \eqref{eq: weight space decomposition} we have
	\begin{align*}
		w([J]_k, \lambda) &= -\sum_{i=0}^k \left(i a \binom{n+i-1}{i} + (k-i) b \psi_Y(k-i) \right) \\
		&= - \sum_{i=1}^k i \left(a \binom{n+i-1}{i} + b \psi_Y(i) \right),
	\end{align*}
	where $\psi_Y(i) := \dim_{\C} (\C[y_1, \dots, y_m]_i / (I_Y)_i)$. For all sufficiently large $k$, we have 
	$$ \psi_Y(k) = h^0(Y, \OO_{\PP(W)}(k)|_Y) = \frac{d}{r!} k^r + O(k^{r-1}). $$ 
	Therefore, for $k \gg 0$,
	\begin{align*}
		w([J]_k, \lambda) &= - \sum_{i=1}^k i \left(a \binom{n+i-1}{i} + b \left( \frac{d}{r!}i^r + O(i^{r-1}) \right) \right) + O(1) \\
		&= - \sum_{i=1}^k i a \binom{n + i - 1}{i} - b \frac{d}{r!(r+2)} k^{r+2} + O(k^{r+1}) \\
		&= \frac{-b d (r+1)}{(r+2)!} k^{r+2} - \sum_{i=1}^k i a \binom{n + i - 1}{i} + O(k^{r+1}),
	\end{align*}
	where in the second line we used $\sum_{i=1}^k i^p = \frac{k^{p+1}}{p+1} + O(k^p)$. On the other hand we have $\sum_{i=1}^k  i a \binom{n + i - 1}{i} = \frac{a n}{(n+1)!} k^{n+1} +  O(k^{n})$. The lemma then follows from Proposition \ref{prop: chow weight general}.
\end{proof}

\section{Moduli of Hyperplanar Admissible Flags} \label{section: moduli of flags and the local universal property}

\subsection{The Moduli Functor} \label{section: the moduli functor}

Fix a finite dimensional vector space $V$ and positive integers $d, n > 0$.

\begin{defn} \label{defn: complete flag}
	A \emph{(complete) hyperplanar admissible flag} $(\underline{X}, \underline{Z})$ of subschemes of $\PP(V)$ (of \emph{length $n$} and \emph{degree $d$}) is a tuple $\underline{X} = (X^0, X^1, \dots, X^n)$ of closed subschemes $X^0 \subset X^1 \subset \cdots \subset X^{n-1} \subset X^n$, together with a tuple $\underline{Z} = (Z^0, Z^1, \dots, Z^n)$ of linear subspaces $Z^0 \subset Z^1 \subset \cdots \subset Z^{n-1} \subset Z^n = V$ such that:
	\begin{enumerate}[label=(\roman*)]
		\item $\dim X^i = i$ for all $i = 0, 1, \dots, n$,
		\item $\codim_{V} Z^i = n-i$ for each $i = 0, 1, \dots, n$, and
		\item $X^i = X^n \cap \PP(Z^i)$ for each $i = 0, 1, \dots, n$.\footnote{In particular, $\deg X^i = d$ for all $i = 0, \dots, n$.}
	\end{enumerate}  
The flag $(\underline{X}, \underline{Z})$ is \emph{non-degenerate} if each $X^i$ is not contained in a proper linear subspace of $Z^i$. The flag $(\underline{X}, \underline{Z})$ is \emph{non-singular} if each of $X^1, \dots, X^n$ is a non-singular connected projective variety, and if $X^0$ is the disjoint union of $d$ reduced points. The flag $(\underline{X}, \underline{Z})$ is \emph{stable} if $X^0 \subset \PP(Z^0)$ is Chow stable, in the sense that for each proper linear subspace $Z \subset \PP(Z^0)$, Inequality \eqref{eq: chow stability for points} holds strictly, with $W$ taken to be $Z^0$ and $Y$ to be $X^0$.
\end{defn}

\begin{remark}
	If $(\underline{X}, \underline{Z})$ is non-degenerate, the linear subspace $Z^i$ can be recovered from $\underline{X}$ by taking the linear span of $X^i$ in $\PP(V)$.
\end{remark}

Of course, length $n$ hyperplanar admissible flags of subschemes of $\PP(V)$ make sense only when $\dim \PP(V) \geq n$, and if $n = \dim \PP(V)$ then the only possibilities are complete flags of linear subspaces, which are all projectively equivalent. As such whenever we discuss length $n$ hyperplanar admissible flags of subschemes of $\PP(V)$, we assume that $n + 1 < \dim V$. A necessary condition for there to exist non-degenerate, non-singular and stable hyperplanar admissible flags of degree $d$ is that $d > \dim Z^0 = \dim V - n$. Indeed, if $k$ of the points of a Chow stable configuration $X^0 \subset \PP(Z^0)$ have a projective linear span of dimension $\ell \leq k-1$, Inequality \ref{eq: chow stability for points} implies that $d > \frac{k}{\ell + 1} \dim Z^0$. 

On the other hand, if $d > \dim Z^0$ then a generic configuration of points $X^0 \subset \PP(Z^0)$ obtained by intersecting a non-degenerate curve $X^1 \subset \PP(Z^1)$ of degree $d$ with a generic hyperplane $\PP(Z^0) \subset \PP(Z^1)$ will be Chow stable, as such a generic intersection will satisfy the property that any subset of $X^0$ of size $\dim \PP(Z^0)$ spans a hyperplane in $\PP(Z^0)$ by the so-called \emph{general position theorem} of \cite[Chapter III, §1]{geometryofalgcurvesvol1}. The non-triviality of the moduli of non-degenerate, non-singular and stable hyperplane admissible flags for such degrees is then implied by Bertini's theorem, together with the non-degenericity of generic hyperplane sections of irreducible non-degenerate projective subvarieties (cf. \cite[Proposition 18.10]{harrisbasicag}).

\begin{defn}
	The \emph{Hilbert type} $\underline{\Phi} = (\Phi_0, \dots, \Phi_n)$ of the flag $(\underline{X}, \underline{Z})$ is the tuple whose entries are the Hilbert polynomials of the $X^i$:
	$$ \Phi_i(t) = \chi(X^i, \OO_{\PP(V)}(t)|_{X^i}). $$
\end{defn}

In order to have a moduli functor, we need to specify what is meant by a \emph{family} of hyperplanar admissible flags, and when two families are considered to be equivalent.

\begin{defn} \label{defn: family of complete flags}
	Let $S$ be a scheme. A \emph{family of hyperplanar admissible flags} $(\underline{\mathcal{X}}, \underline{\mathcal{Z}}, L)$ of subschemes of $\PP(V)$ (of length $n$, degree $d$ and Hilbert type $\underline{\Phi}$) \emph{over $S$} is given by the following data:
	\begin{enumerate}
		\item an invertible sheaf $L$ on $S$,
		\item a tuple $\underline{\mathcal{X}} = (\mathcal{X}^0, \dots, \mathcal{X}^n)$ of closed subschemes $\mathcal{X}^0 \subset \cdots \subset \mathcal{X}^n \subset \PP_S(V \otimes L)$ which are flat and finitely presented over $S$, and
		\item a tuple $\underline{\mathcal{Z}} = (\mathcal{Z}^0, \dots, \mathcal{Z}^n)$ of subbundles $\mathcal{Z}^0 \subset \cdots \subset \mathcal{Z}^n = V \otimes L$,
	\end{enumerate}
such that the following properties hold:
\begin{enumerate}[label=(\roman*)]
	\item each $\mathcal{Z}^i$ is of corank $n-i$, and
	\item for each geometric point $s \in S$, the fibre $(\underline{\mathcal{X}}, \underline{\mathcal{Z}})_s = (\underline{\mathcal{X}}_s, \underline{\mathcal{Z}}|_s)$ is a hyperplanar admissible flag of subschemes of $\PP(V \otimes L|_s) \cong \PP(V)$ of length $n$, degree $d$ and Hilbert type $\underline{\Phi}$, as in Definition \ref{defn: complete flag}.
\end{enumerate}
The family $(\underline{\mathcal{X}}, \underline{\mathcal{Z}}, L)$ is said to be \emph{non-degenerate} (resp. \emph{non-singular}, \emph{stable}) if for each geometric point $s \in S$, the flag $(\underline{\mathcal{X}}, \underline{\mathcal{Z}})_s$ is non-degenerate (resp. non-singular, stable).
\end{defn}

\begin{defn} \label{defn: equivalence of families}
	Let $(\underline{\mathcal{X}}, \underline{\mathcal{Z}}, L)$ and $(\underline{\mathcal{X}}', \underline{\mathcal{Z}}', L')$ be families of hyperplanar admissible flags of subschemes of $\PP(V)$ over a common base scheme $S$. These families are said to be \emph{(projectively) equivalent} if there exists an isomorphism of invertible sheaves $\phi : L \to L'$ and an element $g \in GL(V)$ such that:
	\begin{enumerate}[label=(\roman*)]
		\item $g \otimes \phi : V \otimes L \stackrel{\sim}{\to} V \otimes L'$ sends $\mathcal{Z}^i$ to $(\mathcal{Z}^i)'$ for all $i = 0, \dots, n$, and
		\item $\PP_S(g \otimes \phi) : \PP_S(V \otimes L) \stackrel{\sim}{\to} \PP_S(V \otimes L')$ sends $\mathcal{X}^i$ to $(\mathcal{X}^i)'$ for all $i = 0, \dots, n$.
	\end{enumerate}
\end{defn}

Given a morphism of schemes $T \to S$, the pullback of a family of hyperplanar admissible flags $(\underline{\mathcal{X}}, \underline{\mathcal{Z}}, L)$ parametrised by $S$ is a family of hyperplanar admissible flags parametrised by $T$, and projectively equivalent families over $S$ pull back to projectively equivalent families over $T$. As such, there is a well-defined moduli functor 
$$ \mathcal{F}_{n,d,\underline{\Phi}}^{\PP(V)} : \mathbf{Sch}^{\mathrm{op}} \to \mathbf{Set} $$ which associates to a scheme $S$ the set of all equivalence classes of families of non-degenerate, non-singular and stable hyperplanar admissible flags of subschemes of $\PP(V)$ of length $n$, degree $d$ and Hilbert type $\underline{\Phi}$ parametrised by $S$.

\begin{remark}
	Suppose $(\underline{X}, \underline{Z})$, $(\underline{X}', \underline{Z}')$ are non-degenerate hyperplanar admissible flags of subschemes of $\PP(V)$ with the closed embeddings $X^n, (X')^n \subset \PP(V)$ being linearly normal, i.e. $H^1(\PP(V), \mathcal{I}_{X^n}(1)) = H^1(\PP(V), \mathcal{I}_{(X')^n}(1)) = 0$, so that the restriction maps $V^{\vee} \to H^0(X^n, \OO_{X^n}(1))$ and $V^{\vee} \to H^0((X')^n, \OO_{{X'}^n}(1))$ are both isomorphisms. Then there exists an isomorphism of polarised varieties $(X^n, \OO_{\PP(V)}(1)|_{X^n}) \cong ((X')^n, \OO_{\PP(V)}(1)|_{(X')^n})$ taking each $X^i$ to $(X')^i$ if and only if $(\underline{X}, \underline{Z})$ and $(\underline{X}', \underline{Z}')$ are projectively equivalent. Indeed, any such isomorphism $(X^n, \OO_{\PP(V)}(1)|_{X^n}) \cong ((X')^n, \OO_{\PP(V)}(1)|_{(X')^n})$ gives rise to an automorphism of the overlying projective space $\PP(V)$, by considering the induced isomorphism on global sections.
\end{remark}

\subsection{The Local Universal Property} \label{section: local universal property}

Let us recall the following terminology from moduli theory.

\begin{defn} \label{defn: cms}
	Let $\mathcal{F} : \mathbf{Sch}_{\C}^{\mathrm{op}} \to \mathbf{Set}$ be a moduli functor.
	\begin{enumerate}
		\item If $S$ is a scheme and if $Y \in \mathcal{F}(S)$, $Y$ is said to have the \emph{local universal property} for $\mathcal{F}$ if for any scheme $S'$, any object $Y' \in \mathcal{F}(S')$ and any point $s' \in S'$, there exists an open neighbourhood $U \subset S'$ of $s'$ and a morphism $f : U \to S$ such that $f^{\ast} Y \cong Y'|_U$.
		\item A pair $(M, \eta)$ consisting of a scheme $M$ and a natural transformation $\eta : \mathcal{F} \to \mathrm{Hom}(-,M)$ is said to \emph{corepresent} $\mathcal{F}$ if for any scheme $M'$ and any natural transformation $\eta' : \mathcal{F} \to \mathrm{Hom}(-,M')$, there exists a unique morphism of schemes $f : M \to M'$ making the diagram
		\[\begin{tikzcd}
			{\mathcal{F}} && {\mathrm{Hom}(-,M)} \\
			& {\mathrm{Hom}(-,M')}
			\arrow["\eta", from=1-1, to=1-3]
			\arrow["{\eta'}"', from=1-1, to=2-2]
			\arrow["{f \circ -}", from=1-3, to=2-2]
		\end{tikzcd}\]
	commute. If in addition $\eta_{\C} : \mathcal{F}(\spec \C) \to M(\C)$ is a bijection, then $(M, \eta)$ is known as a \emph{coarse moduli space} of $\mathcal{F}$.
	\end{enumerate}
\end{defn}

\begin{prop}[\cite{newstead}, Proposition 2.13] \label{prop: existence of corepresentations}
	Let $\mathcal{F} : \mathbf{Sch}_{\C}^{\mathrm{op}} \to \mathbf{Set}$ be a moduli functor. Suppose $Y \in \mathcal{F}(S)$ has the local universal property for $\mathcal{F}$, and suppose moreover that there exists an algebraic group $H$ acting on $S$ with the property that for any two points $s, t \in S$, $Y_s \cong Y_t$ if and only if $s$ and $t$ lie in the same $H$-orbit.
	\begin{enumerate}
		\item A scheme $M$ corepresents $\mathcal{F}$ if and only if $M$ is a categorical quotient of $S$ by $H$.
		\item A categorical quotient $M$ of $S$ by $H$ is a coarse moduli space of $\mathcal{F}$ if and only if $M$ is an orbit space.
	\end{enumerate}
\end{prop}

Let $n, d$ be positive integers and let $V$ be a complex vector space of dimension $\dim V > n + 1$. We now exhibit a family which has the local universal property for the moduli functor $\mathcal{F} = \mathcal{F}_{n,d,\underline{\Phi}}^{\PP(V)}$. Let $\mathcal{H}_i = \mathrm{Hilb}(\PP(V), \Phi_i)$ be the Hilbert scheme parametrising closed subschemes of $\PP(V)$ with Hilbert polynomial $\Phi_i$, and let $\Gr_i = \Gr_i(V)$ be the Grassmannian parametrising subspaces of $V$ of \emph{codimension} $n-i$. Let
$$ \mathcal{S} := \prod_{i=0}^n \mathcal{H}_i \times \prod_{j=0}^n \Gr_j. $$
We endow $\mathcal{S}$ with the diagonal action of $GL(V)$ induced by the natural $GL(V)$-actions on the $\mathcal{H}_i$ and $\Gr_j$. Over $\mathcal{H}$ are the following universal objects:
\begin{enumerate}[label=(\roman*)]
	\item $\mathcal{S}$-flat closed subschemes of finite presentation $\mathcal{Y}^i \subset \PP(V) \times \mathcal{S}$ for each $i = 0, \dots, n$, with $\mathcal{Y}^i$ corresponding to the universal family over $\mathcal{H}_i$.
	\item Corank $n-j$ subbundles $\mathcal{W}^j \subset V \otimes \OO_{\mathcal{S}}$ for each $j = 0, \dots, n$, with $\mathcal{W}^j$ corresponding to the universal subbundle over $\Gr_j$.
\end{enumerate}

\begin{lemma} \label{lemma: universal property flag parameter space}
	There exists a locally closed, $GL(V)$-invariant subscheme $\mathcal{S}' \subset \mathcal{S}$ with the following universal property: suppose $S$ is a scheme and $S \to \mathcal{S}$ is a morphism, corresponding to the closed subschemes $\mathcal{Y}_S^i \subset \PP(V) \times S$ and the locally free subsheaves $\mathcal{W}_S^j \subset V \otimes \OO_S$. Then $S \to \mathcal{S}$ factors through $\mathcal{S}'$ if and only if for each geometric point $s \in S$:
	\begin{enumerate}
		\item $\mathcal{Y}^0_s \subset \cdots \subset \mathcal{Y}^n_s$ as subschemes of $\PP(V)$, and $\mathcal{W}^0_s \subset \cdots \subset \mathcal{W}^n_s = V$ as subspaces of $V$;
		\item $\mathcal{Y}^i_s = \mathcal{Y}^n_s \cap \PP(\mathcal{W}^i_s)$ for all $i = 0, \dots, n$;
		\item $\dim \mathcal{Y}^i_s = i$ for all $i = 0, \dots, n$;
		\item $\deg \mathcal{Y}_s^i = d$ for all $i = 0, \dots, n$;
		\item $\mathcal{Y}^i_s \subset \PP(\mathcal{W}^i_s)$ is a non-degenerate, non-singular, connected projective subvariety for each $i = 1, \dots, n$; and
		\item $\mathcal{Y}^0_s \subset \PP(\mathcal{W}^0_s)$ is non-degenerate, $\mathcal{Y}^0_s$ is the disjoint union of $d$ reduced points and $\mathcal{Y}^0_s$ is Chow stable.
	\end{enumerate}
\end{lemma}

\begin{proof}
	The first two conditions are incidence correspondences; these conditions cut out a closed subscheme of $\mathcal{S}$. After first imposing the first two conditions, the remaining conditions are all open in flat families, so define an open subscheme $\mathcal{S}'$ of this closed subscheme.
\end{proof}

Let $(\underline{\mathcal{Y}}', \underline{\mathcal{W}}')$ be the restriction of $(\underline{\mathcal{Y}}, \underline{\mathcal{W}})$ to $\mathcal{S}' \subset \mathcal{S}$. Note that $(\underline{\mathcal{Y}}', \underline{\mathcal{W}}', \OO_{\mathcal{S}'})$ is a family of non-degenerate, non-singular and stable hyperplanar admissible flags of subschemes of $\PP(V)$ of length $n$, degree $d$ and Hilbert type $\underline{\Phi}$ over $\mathcal{S}'$.

\begin{lemma} \label{lemma: local universal property}
	The family $(\underline{\mathcal{Y}}', \underline{\mathcal{W}}', \OO_{\mathcal{H}'})$ has the local universal property for the moduli functor $\mathcal{F} = \mathcal{F}_{n,d,\underline{\Phi}}^{\PP(V)}$.
\end{lemma}

\begin{proof}
	Given a family $(\underline{\mathcal{X}}, \underline{\mathcal{Z}}, L)$ over a scheme $S$ and a point $s \in S$, by passing to an open neighbourhood $U \subset S$ of $s$ where $L$ is trivial we may assume without loss of generality that $L = \OO_S$. The result then follows from Lemma \ref{lemma: universal property flag parameter space} together with the universal properties of the schemes $\mathcal{H}_i$ and $\mathrm{Gr}_j$.
\end{proof}

\begin{cor} \label{cor: parameter space to be quotiented}
	The moduli functor $\mathcal{F}$ is corepresentable if and only if there exists a categorical quotient $q : \mathcal{S}' \to \mathcal{S}' \sslash SL(V)$ of $\mathcal{S}'$ by $SL(V)$, and that a coarse moduli space of $\mathcal{F}$ exists if and only if $q$ is an orbit space morphism.
\end{cor}

\begin{proof}
	If $s_0$, $s_1$ are points of $\mathcal{S}'$ and if $(\underline{\mathcal{Y}}_0, \underline{\mathcal{W}}_0)$, $(\underline{\mathcal{Y}}_1, \underline{\mathcal{W}}_1)$ are the restrictions of the universal family $(\underline{\mathcal{Y}}, \underline{\mathcal{W}})$ to $s_0$ and $s_1$ respectively, the flags $(\underline{\mathcal{Y}}_0, \underline{\mathcal{W}}_0)$ and $(\underline{\mathcal{Y}}_1, \underline{\mathcal{W}}_1)$ are equivalent if and only if $s_0$ and $s_1$ lie in the same $SL(V)$-orbit in $\mathcal{S}'$. The result then follows by Proposition \ref{prop: existence of corepresentations}.
\end{proof}

\section{Non-Reductive Geometric Invariant Theory} \label{section: NRGIT}

\subsection{The $\hat{U}$-Theorem}

Here we give a summary of the non-reductive GIT (NRGIT) machinery used in this paper. Let $H = U \rtimes L$ be a linear algebraic group with unipotent radical $U$ and where $L$ is a reductive Levi subgroup. We are interested in groups $H$ with an \emph{(internally) graded} unipotent radical.

\begin{defn}
	An \emph{(internal) grading of the unipotent radical $U$ of $H$} is a 1PS $\lambda : \G_m \to Z(L)$ such that, under the conjugation action, $\lambda(\G_m)$ acts on $\Lie U$ with strictly positive weights. If $\lambda$ is a choice of a grading 1PS, we denote
	$$ \hat{U} := U \rtimes \lambda(\G_m) \quad \text{and} \quad \overline{L} := L/\lambda(\G_m). $$
\end{defn}

Fix an internal grading $\lambda : \G_m \to Z(L)$ of the unipotent radical. Let $Y$ be a projective scheme acted on by $H$, and let $\L$ be an ample linearisation for this action.	Let $\omega_{\min} = \omega_0 < \omega_1 < \cdots < \omega_{\max}$ be the weights with which $\lambda(\G_m)$ acts on the fibres of $\L^{\vee}$ over points of the connected components of the fixed-point locus $Y^{\lambda(\G_m)}$ for the action of $\lambda(\G_m)$ on $Y$. For simplicity, we will assume that there is a dense open subscheme $Y^{\circ} \subset Y$ such that for all points $y \in Y^{\circ}$, $\lambda(\G_m)$ acts on $\L^{\vee}|_{p(y)}$ with weight $\omega$, where $p(y) = \lim_{t \to 0} \lambda(t) \cdot y$; note that this implies $\omega = \omega_{\min}$.

Choose a positive integer $N$ such that $\L^N$ is very ample, so that there is a closed immersion $Y \subset \PP(W)$, where $W = H^0(Y, \L^N)^{\vee}$. Let $W_{\min}$ be the weight space of weight $N\omega_{\min}$ in $W$. We introduce the following loci:
\begin{enumerate}[label=(\roman*)]
	\item The \emph{$\lambda$-minimal weight space} is the closed subscheme of $Y$
	$$ Z_{\min} = Z(Y, \lambda)_{\min} := Y \cap \PP(W_{\min}). $$ 
	The scheme $Z_{\min}$ admits a natural $L$-action. The set of closed points of $Z_{\min}$ is given by $\left\{y \in Y^{\lambda(\G_m)} : \lambda(\G_m) \text{ acts on } \L^{\vee}|_y \text{ with weight } N\omega_{\min} \right\}$.
	\item The \emph{$\lambda$-attracting open set} is the open Białynicki-Birula stratum
	$$ Y_{\min} = Y(\lambda)_{\min} = \left\{y \in Y : \lim_{t \to 0} \lambda(t) \cdot y \in Z_{\min} \right\}. $$
	This is naturally acted on by $H$.
	\item The \emph{$\lambda$-retraction} is the morphism $p = p_{\lambda} : Y_{\min} \to Z_{\min}$ given (on closed points) by $p(y) = \lim_{t \to 0} \lambda(t) \cdot y$. This morphism is equivariant with respect to the quotient $H \to L$ (cf. \cite[Lemma 4.16]{bkmomentmaps}); in particular, any point of $U Z_{\min}$ is fixed by $p$.
	\item The \emph{$U$-very stable locus}\footnote{Note that the terminology here is non-standard and is not present in e.g. \cite{bdhkconstructing} \cite{bdhkapplications} \cite{bkmomentmaps} \cite{hoskinsjackson}. Here we wish to emphasise that we only consider points whose $\lambda(\G_m)$-limits have trivial $U$-stabilisers, and do not require that all points of $Z_{\min}$ or $Z_{\min}^{\overline{L}-s}$ have trivial $U$-stabilisers.} is the open subscheme of $Y$ whose set of points is
	$$ Y_{\min}^{U-vs} := \{ y \in Y_{\min} : \stab_U(p(y)) = \{e\} \}. $$
	\item The \emph{$H$-very stable locus} is the open subscheme of $Y$ whose set of points is
	$$ Y_{\min}^{H-vs} := \{ y \in Y_{\min}^{U-vs} : p(y) \in Z_{\min}^{\overline{L}-s} \} \setminus U Z_{\min}. $$
\end{enumerate}

\begin{remark}
	That $Y_{\min}^{U-vs}$ is open in $Y_{\min}$ follows from the semicontinuity of dimensions of $U$-stabiliser dimensions. In turn, the same argument given in \cite[Proof of Theorem 4.21, Page 21]{bkmomentmaps} shows that the $U$-sweep of $\{z \in Z_{\min} : \dim \stab_U(z) = 0\}$ is closed in $Y_{\min}^{U-vs}$, whence $Y_{\min}^{H-vs}$ is an open subscheme of $Y_{\min}$. $Y_{\min}^{U-vs}$ is invariant under the action of $U$, and $Y_{\min}^{H-vs}$ is invariant under the action of $H$.
\end{remark}

We also introduce the notion of an \emph{adapted linearisation}.

\begin{defn}
	The ample linearisation $\L$ on $Y$ is said to be \emph{adapted} for the action of $\hat{U} \subset H$ if $\omega_{\min} = \omega_0 < 0 < \omega_1 < \cdots < \omega_{\max}$.
\end{defn}

Note that any positive integer power of an adapted linearisation is adapted. Moreover, the schemes $Y_{\min}^{U-vs}$ and $Y_{\min}^{H-vs}$ only depend on $\L$ and not the very ample linearisation $\L^N$. As such, when trying to form quotients of the schemes $Y_{\min}^{U-vs}$ and $Y_{\min}^{H-vs}$, nothing is lost by assuming $\L$ itself is very ample and that $N = 1$.

The version of the \emph{$\hat{U}$-theorem} from NRGIT we will make use of is the following result.

\begin{theorem} \label{theorem: U-hat in our setting}
	Let $H = U \rtimes L$ be a linear algebraic group with internal grading $\lambda : \G_m \to Z(L)$. Let $Y$ be a projective scheme acted on by $H$ and let $\L$ be a very ample adapted linearisation on $Y$. Assume that there is a dense open subscheme $Y^{\circ} \subset Y$ such that for all points $y \in Y^{\circ}$, $\lambda(\G_m)$ acts on $\L^{\vee}|_{p(y)}$ with common weight $\omega$. Assume in addition that $Y_{\min}^{U-vs}$ is non-empty.
	\begin{enumerate}
		\item There exists a locally trivial $U$-quotient $q_U : Y_{\min}^{U-vs} \to Y_{\min}^{U-vs}/U$. Moreover, there exists a positive integer $m > 0$ such that $Y_{\min}^{U-vs}/U$ admits a locally closed immersion into $\PP((H^0(Y, \L^m)^U)^{\vee})$ in such a way that $q_U$ is induced by the linear projection $\PP(H^0(Y, \L^m)^{\vee}) \dashrightarrow \PP((H^0(Y, \L^m)^{U})^{\vee})$.
		\item Assume further that $Y_{\min}^{H-vs}$ is non-empty. Fixing such an integer $m$, let $Q$ denote the closure of $Y_{\min}^{U-vs}/U$ in $\PP((H^0(Y, \L^m)^{U})^{\vee})$. Then there exists a map
		\[\begin{tikzcd}
			{q : Y_{\min}^{U-vs}} && {Y_{\min}^{U-vs}/U} && {Q \sslash L}
			\arrow["{q_U}", from=1-1, to=1-3]
			\arrow["{q_L}", dashed, from=1-3, to=1-5]
		\end{tikzcd}\]
		with the following properties:
		\begin{enumerate}
			\item The locus $Y_{\min}^{H-vs}$ is contained in the domain of $q$, and the restriction $q : Y_{\min}^{H-vs} \to q(Y_{\min}^{H-vs})$ is a geometric $H$-quotient.
			\item There exists an integer $M \geq m$ and a locally closed immersion of $q(Y_{\min}^{H-vs})$ into $\PP((H^0(Y, \L^M)^H)^{\vee})$ in such a way that $q$ is induced by the linear projection $\PP(H^0(Y, \L^M)^{\vee}) \dashrightarrow \PP((H^0(Y, \L^M)^H)^{\vee})$.
		\end{enumerate}
	\end{enumerate}
\end{theorem}

\begin{remark}
	It follows from Lemmas \ref{lemma: restricting geometric quotients reductive} and \ref{lemma: restricting geometric quotients locally trivial} that if $Z$ is any $H$-invariant locally closed subscheme of $Y_{\min}^{H-vs}$ then $q$ restricts to give a geometric $H$-quotient $Z \to q(Z)$, with $q(Z) = Z/H$ a locally closed subscheme of $q(Y) = Y_{\min}^{H-vs}/H$.
\end{remark}

\begin{proof}[Proof of Theorem \ref{theorem: U-hat in our setting}]
	This theorem can be deduced using the results presented in \cite{qiaothesis}, however we will give another proof.
	
	We begin by first forming the quotient by $U$. Let $W = H^0(Y, \L)^{\vee}$. As $Y \subset \PP(W)$ is not contained in any proper linear subspace, the minimal weight for the $\lambda(\G_m)$-action on the fibres of $\OO_{\PP(W)}(-1)$ over points in $\PP(W)^{\lambda(\G_m)}$ is also equal to $\omega$, so there is an equality of schemes $Y_{\min}^{U-vs} = Y \cap \PP(W)_{\min}^{U-vs}$. By \cite[Proposition 4.26]{bkmomentmaps}, there is a locally trivial $U$-quotient $\PP(W)_{\min}^{U-vs} \to \PP(W)_{\min}^{U-vs}/U$, given by restricting an enveloping quotient map. Moreover, from the proof of \cite[Proposition 3.1.19]{bdhkconstructing} there exists a positive integer $m > 0$ such that for all positive integer multiples $m'$ of $m$, $\PP(W)_{\min}^{U-vs}/U$ admits a locally closed immersion to $\PP((H^0(\PP(W), \OO_{\PP(W)}(m'))^U)^{\vee})$, with the composite $\PP(W)_{\min}^{U-vs} \to \PP(W)_{\min}^{U-vs}/U \to \PP((H^0(\PP(W), \OO_{\PP(W)}(m'))^U)^{\vee})$ being the morphism obtained by restricting the rational map between projective spaces associated to the inclusion $H^0(\PP(W), \OO_{\PP(W)}(m'))^{U} \subset H^0(\PP(W), \OO_{\PP(W)}(m'))$.
	
	By Lemma \ref{lemma: restricting geometric quotients locally trivial}, this quotient restricts to give a locally trivial quotient $q_U : Y_{\min}^{U-vs} \to Y_{\min}^{U-vs}/U$, with $Y_{\min}^{U-vs}/U$ a closed subscheme of $\PP(W)_{\min}^{U-vs}/U$. After replacing $m$ with a larger multiple if necessary (so that the natural map $\mathrm{Sym}^m H^0(Y, \L) \to H^0(Y, \L^m)$ is onto), the closed immersion $Y \to \PP(H^0(\PP(W), \OO_{\PP(W)}(m))^{\vee})$ factors through $\PP(H^0(Y, \L^m)^{\vee})$. In addition, the composite $H^0(\PP(W), \OO_{\PP(W)}(m))^U \stackrel{\subset}{\to} H^0(\PP(W), \OO_{\PP(W)}(m)) \to H^0(Y, \L^{m})$ factors through the subspace $H^0(Y, \L^{m})^U \subset H^0(Y, \L^{m})$. It follows that the composite $Y_{\min}^{U-vs}/U \to \PP(W)_{\min}^{U-vs}/U \to \PP((H^0(\PP(W), \OO_{\PP(W)}(m))^U)^{\vee})$ factors through the projective space $\PP((H^0(Y, \L^m)^U)^{\vee})$, and so $q_U$ fits into a diagram of the form
	\[\begin{tikzcd}[ampersand replacement=\&]
		{Y_{\min}^{U-vs}} \&\& {\PP(H^0(Y, \L^{m})^{\vee})} \\
		\\
		{Y_{\min}^{U-vs}/U} \&\& {\PP((H^0(Y, \L^{m})^U)^{\vee})}
		\arrow["{H^0(Y, \L^{m})^U \subset H^0(Y, \L^{m})}", dashed, from=1-3, to=3-3]
		\arrow["{q_U}"', from=1-1, to=3-1]
		\arrow[from=1-1, to=1-3]
		\arrow[from=3-1, to=3-3]
	\end{tikzcd}\]
	whose horizontal arrows are given by locally closed immersions. This completes the proof of the first part of the theorem.
	
	The proof of the second part of the theorem now boils down to reductive GIT for the residual action of $L$. Let $Q$ be the closure of $Y_{\min}^{U-vs}/U$ in $\PP((H^0(Y, \L^m)^{U})^{\vee})$, and endow $Q$ with the very ample linearisation $\L_Q$ obtained by taking the restriction of the $\OO(1)$ of $\PP((H^0(Y, \L^m)^{U})^{\vee})$. Choosing $s > 0$ sufficiently large such that the finitely generated graded algebra $R(Q, \L_Q^s)^{L} = \bigoplus_{d \geq 0} H^0(Q, \L_Q^{sd})^{L}$ is generated in degree 1 and such that the natural morphism
	\[\begin{tikzcd}[ampersand replacement=\&]
		{\mathrm{Sym}^rH^0(Y, \L^m)^U} \&\& {H^0(Q, \L_Q^{r})} \&\& 0
		\arrow[from=1-1, to=1-3]
		\arrow[from=1-3, to=1-5]
	\end{tikzcd}\]
	is a surjection for all $r \geq s$, standard results of reductive GIT yields that the map
	$$ q_{L} : Q \dashrightarrow Q \sslash L = \mathrm{Proj} R(Q, \L_Q^s)^{L} \subset \PP(( H^0(Q, \L_Q^{s})^{L} )^{\vee}) $$ 
	induced by the inclusion $R(Q, \L_Q^s)^{L} \subset R(Q, \L_Q^s)$ restricts to give a geometric $L$-quotient of the open stable locus $Q^{L-s}(\mathcal{L}_Q)$. Increasing $s$ if necessary, the above morphisms give rise to surjections $H^0(Y, \L^{rm})^U \to H^0(Q, \L_Q^r)$. As taking invariants of the reductive group $L$ is exact, there are induced surjections $(H^0(Y, \L^{rm})^U)^L = H^0(Y, \L^{rm})^H \to H^0(Q, \L_Q^r)^L$, so we have a closed immersion $\PP(( H^0(Q, \L_Q^{s})^{L} )^{\vee}) \to \PP((H^0(Y, \L^{sm})^H)^{\vee})$ such that the composite $Q \sslash L \to \PP((H^0(Y, \L^{sm})^{H})^{\vee})$ fits into the following diagram:
	\[\begin{tikzcd}[ampersand replacement=\&]
		{Y_{\min}^{U-vs}} \&\& {\PP(H^0(Y, \L^{sm})^{\vee})} \\
		\\
		{Y_{\min}^{U-vs}/U} \\
		\\
		{Q \sslash \lambda(\G_m)} \&\& {\PP((H^0(Y, \L^{sm})^{H})^{\vee})}
		\arrow["{q_U}"', from=1-1, to=3-1]
		\arrow["{q_L}"', dashed, from=3-1, to=5-1]
		\arrow[from=1-1, to=1-3]
		\arrow[from=5-1, to=5-3]
		\arrow["{H^0(\L^{sm})^{H} \subset H^0(\L^{sm})}", dashed, from=1-3, to=5-3]
	\end{tikzcd}\]
	
	In light of Lemmas \ref{lemma: restricting geometric quotients reductive} and \ref{lemma: restricting geometric quotients locally trivial}, it remains to show that the locally closed subscheme $q_U(Y_{\min}^{H-vs})$ of $Q$ is contained in the open stable locus $Q^{L-s}(\mathcal{L}_Q)$; that is, for each point $y \in Y_{\min}^{H-vs}$, the point $q_U(y)$ is stable with respect to the residual $L$-action on $Q$ and the linearisation $\L_Q$. But this follows from the first assertion of \cite[Theorem 2.28]{hoskinsjackson}.
\end{proof}

\subsection{Parabolic Subgroups of $SL(V)$} \label{section: parabolic subgroups of SL(V)}

Here we introduce the necessary notation involved when carrying out a quotienting-in-stages procedure for the action of a parabolic subgroup $P \subset SL(V)$ on a projective scheme $Y$.

Let $\lambda : \G_m \to SL(V)$ be a 1PS. Fix a basis $\{v_1, \dots, v_N\}$ for $V$ diagonalising the action of $\lambda(\G_m)$, with weights $r_1 \geq r_2 \geq \cdots \geq r_N$. Let $\ell = \ell(\lambda)$ be the number of distinct weights of $\lambda$, let $\beta_1 > \cdots > \beta_{\ell}$ be these distinct weights, and suppose the weight $\beta_i$ occurs $m_i$-times.

Let
$$ P = P(\lambda) = \left\{ g \in SL(V) : \lim_{t \to 0} \lambda(t) g \lambda(t)^{-1} \text{ exists in } SL(V) \right\} = U \rtimes L $$
be the parabolic subgroup of $SL(V)$ associated to $\lambda$, with unipotent radical $U = U(\lambda)$ and Levi factor $L = L(\lambda)$. As stated in \cite[Section 4.2]{hoskinsjackson}, the groups $P$, $U$ and $L$ can be explicitly described in terms of block upper-triangular matrices (defined with respect to the basis $\{v_1, \dots, v_N\}$) as follows:
\begin{enumerate}[label=(\roman*)]
	\item $P = \left\{ A = \begin{pmatrix}
		A_{11} & A_{12} & \cdots & A_{1, \ell - 1} & A_{1, \ell} \\
		0 & A_{22} & \cdots & A_{2, \ell-1} & A_{2, \ell} \\
		\vdots & \vdots & \ddots & \vdots & \vdots \\
		0 & 0 & \cdots & A_{\ell - 1, \ell - 1} & A_{\ell - 1, \ell} \\
		0 & 0 & \cdots & 0 & A_{\ell, \ell} 
	\end{pmatrix} \in SL(V) : A_{ij} \in \C^{m_i \times m_j} \right\}$.
\item $L = \left\{ A \in P : A_{ij} = 0 \text{ for all } 1 \leq i < j \leq \ell \right\}$ consists of all block diagonal matrices in $P$.
\item The centre $T = Z(L)$ of $L$ is given by
$$ T = \left\{ \mathrm{diag}(t_1 I_{m_1}, \dots, t_{\ell} I_{m_{\ell}}) : t_i \in \G_m, \ \prod_{i=1}^{\ell} t_i = 1 \right\} \cong \G_m^{\ell-1}. $$
\item The semisimple part of $L$ is given by 
$$ R = \left\{ \mathrm{diag}(A_{11}, \dots, A_{\ell, \ell}) \in L : \text{each } A_{ii} \in SL(m_i, \C) \right\}. $$
\item $U = \left\{ A \in P : A_{ii} = I_{m_i} \text{ for all } 1 \leq i \leq \ell \right\}$.
\end{enumerate}

The 1PS $\lambda$ grades the unipotent radical $U$. Let $H := U \rtimes R \subset P$ and $\hat{H} := H \rtimes T = U \rtimes (R \times T)$. There are surjections $R \to L/T$ and $H \to P/T$, both with finite kernels. Since GIT stability is unaffected by finite group actions, forming a quotient of an action by $P$ is equivalent to forming a quotient by an action of $\hat{H}$.

Following \cite[Definition 4.6]{hoskinsjackson}, we introduce the following notation.
\begin{enumerate}[label=(\roman*)]
	\setcounter{enumi}{5}
	\item For each $1 \leq i < \ell$, set $m_{>i} = \sum_{j > i} m_j$, $m_{\leq i} = \sum_{j \leq i} m_j$,
	$$ \beta_{>i} = \frac{\sum_{j > i} \beta_j m_j}{m_{>i}} \quad \text{and} \quad \beta_{\leq i} = \frac{\sum_{j \leq i} \beta_j m_j}{m_{\leq i}}. $$
	We define 1PS $\lambda^{(i)}$ and $\lambda^{[i]}$ of $T = Z(L)$ by setting
	$$ \lambda^{(i)}(t) := \mathrm{diag}(t^{\beta_1} I_{m_1}, \dots, t^{\beta_i} I_{m_i}, t^{\beta_{>i}} I_{m_{>i}}), \quad \lambda^{[i]}(t) = \mathrm{diag}(t^{\beta_{\leq i}} I_{m_{\leq i}}, t^{\beta_{>i}} I_{m_{>i}}). $$
	\item Let $U^{[i]}$ be the unipotent radical of the parabolic $P^{[i]} := P(\lambda^{[i]}) \supset P$:
	$$ U^{[i]} = \{A \in U : \text{for all } p < q, A_{pq} = 0 \text{ if } q \leq i \text{ or } p > i \}. $$
	This is graded by the 1PS $\lambda^{[i]}$ and is normal in $P$.
	\item Let $U^{(i)}$ be the unipotent radical of the parabolic $P(\lambda^{(i)}) \supset P$; this is graded by the 1PS $\lambda^{(i)}$ and is normal in $P$.
	\item Let $P^{(i)} := U^{(i)} \rtimes L^{(i)} \subset P$, where
	$$ L^{(i)} = \begin{cases} \{A \in L : A_{jj} = I_{m_j} \text{ if } j > i \} & \text{if } i < \ell - 1, \\ L & \text{if } i = \ell - 1. \end{cases} $$
	We denote the successive quotients by $L_i := L^{(i)}/L^{(i-1)}$, $U_i := U^{(i)}/U^{(i-1)}$ and $P_i := P^{(i)}/P^{(i-1)}$.
	\item Let $H^{(i)} := U^{(i)} \rtimes R^{(i)} \subset P^{(i)}$, where
	$$ R^{(i)} = \begin{cases} \{A \in R : A_{jj} = I_{m_j} \text{ if } j > i \} & \text{if } i < \ell - 1, \\ R & \text{if } i = \ell - 1. \end{cases} $$
	Let $R_i := R^{(i)}/R^{(i-1)}$ and $H_i := H^{(i)}/H^{(i-1)} = U_i \rtimes R_i$.
	\item Let $T^{(i)} = \prod_{j \leq i} \lambda^{(j)}(\G_m) = \prod_{j \leq i} \lambda^{[j]}(\G_m) \subset T$, and set $\hat{H}^{(i)} := H^{(i)} \rtimes T^{(i)}$.
	\item For each $j \leq i$ let $\lambda_j^{[i]}$ be the 1PS of length $2$ given by the composition
\[\begin{tikzcd}
	{\lambda_j^{[i]} : \G_m} && P && {P/P^{(j-1)},}
	\arrow["{\lambda^{[i]}}", from=1-1, to=1-3]
	\arrow[from=1-3, to=1-5]
\end{tikzcd}\]
Note that $\lambda^{[i]} = \lambda_1^{[i]}$ for all $i$.
\item As a special case of the above, set $\lambda_i = \lambda_i^{[i]}$ for each $i$. The associated parabolic $P(\lambda_i)$ has unipotent radical isomorphic to $U_i$, and $\lambda_i$ grades $U_i$. Let $\hat{H}_i := \hat{H}^{(i)}/\hat{H}^{(i-1)} \cong H_i \rtimes \lambda_i(\G_m) = U_i \rtimes (R_i \times \lambda_i(\G_m))$.
\end{enumerate}

Suppose $\mathcal{Q}_i$ is a projective scheme acted on by $\hat{H}_i$ with a very ample linearisation $\L_i$. The group $\hat{H}_i$ has a one-dimensional character group, which is generated by a character $\chi_i$ dual to the 1PS $\lambda_i = \lambda_i^{[i]}$ grading the unipotent radical $U_i$ of $\hat{H}_i$, and twisting $\L_i$ by some multiple $\epsilon_i \chi_i$ of $\chi_i$ corresponds to shifting the $\lambda_i(\G_m)$-weights each by $\epsilon_i$. As such, by twisting by $\epsilon_i \chi_i$ for suitable $\epsilon_i$, we can always ensure that very ample $\hat{H}_i$-linearisations are adapted; note that twisting by such a character does not affect which points of $\mathcal{Q}_i$ are in the loci $\mathcal{Q}_i(\lambda_i)_{\min}$ or $Z(\mathcal{Q}_i, \lambda_i)_{\min}$.

\section{Proof of the Main Result} \label{section: proof of main result}

We are now ready to begin the proof of Theorem \ref{theorem: main theorem}. Let $n, d$ be positive integers and let $V$ be a complex vector space of dimension $n + 1 < \dim V$. Assume in addition that $d > \dim V - n$ and that $d \not \in \left\{ \frac{\dim V - n - 1 + i}{n + 1 - i} : i = 1, \dots, n \right\}$. Let $\underline{\Phi} = (\Phi_0, \dots, \Phi_n)$ be a tuple of Hilbert polynomials of subschemes of $\PP(V)$, where $\Phi_i(t) = \frac{d}{i!} t^i + O(t^{i-1})$ for each $i = 0, 1, \dots, n$.

\subsection{Reduction to a Parabolic Action} \label{section: reduction to P}

As explained in Section \ref{section: local universal property}, we wish to construct a categorical quotient for the action of the group $G = SL(V)$ on the subscheme $\mathcal{S}' \subset \prod_{i=0}^n \mathcal{H}_i \times \prod_{j=0}^n \Gr_j$. Fix a decomposition
$$ V = W \oplus \C v_1 \oplus \cdots \oplus \C v_n $$
where $W \subset V$ is a subspace of codimension $n$. For each $i = 0, 1, \dots, n$ set $Z^i = W \oplus \bigoplus_{j \leq i} \C v_j$. Let $p_0$ be the point of $\prod_{i=0}^n \Gr_i$ corresponding to $(Z^0, \dots, Z^n)$, and let $P = \stab_G(p_0)$ be the parabolic subgroup of $G$ preserving the flag $Z^0 \subset \cdots \subset Z^n$. In the notation of Section \ref{section: parabolic subgroups of SL(V)} we have $P = P(\tilde{\lambda})$, where $\tilde{\lambda}$ is any 1PS of $G$ with decreasing weights $\beta_1 > \beta_2 > \cdots > \beta_{n+1}$ whose multiplicities are $(\dim W, 1, \dots, 1)$. Let $\mathcal{S}_0' \subset \mathcal{S}'$ be the preimage of $p_0$. Then $\mathcal{S}' = G \cdot \mathcal{S}'_0$, so by Lemma \ref{lemma: borel construction} we have
$$ \mathcal{S}' \cong G \times^P \mathcal{S}_0'. $$
It follows that a categorical quotient of $\mathcal{S}'$ by $G$ exists (and is an orbit space morphism) if and only if a categorical quotient of $\mathcal{S}'_0$ by $P$ exists (and is an orbit space morphism).\footnote{If $\mathcal{S}_0' \sslash P$ exists, then the projection $G \times \mathcal{S}_0' \to \mathcal{S}'_0$ induces an isomorphism of categorical quotients $\mathcal{S}' \sslash G \cong \mathcal{S}_0' \sslash P$.} 

\subsection{Linearising the Parabolic Action}

From now on, we regard $\mathcal{S}_0'$ as a locally closed $P$-invariant subscheme of $\mathcal{H} := \prod_{i=0}^n \mathcal{H}_i$ in the obvious way. For each $i = 0, \dots, n$, let
$$ \Psi_i : \mathrm{Hilb}(\PP(V), \Phi_i) \to \mathrm{Ch}_i := \mathrm{Chow}_{i,d}(\PP(V)) $$
be the corresponding Hilbert--Chow morphism (cf. Section \ref{section: chow linearisation}). Let
$$ \Psi = \prod_{i=0}^n \Psi_i : \mathcal{H} \to \mathrm{Ch} := \prod_{i=0}^n \mathrm{Ch}_i $$
be the product of the $\Psi_i$. Given a tuple of positive integers $\underline{a} = (a_0, \dots, a_n)$, let $\LL_{\mathrm{Ch}}^{\underline{a}} := \boxtimes_{i=0}^n \LL_{\mathrm{Ch}_i}^{a_i}$, where $\LL_{\mathrm{Ch}_i}$ is the Chow linearisation on the Chow scheme $\mathrm{Ch}_i$.

In the notation of Section \ref{section: parabolic subgroups of SL(V)} we have $R_1 = R^{(1)} = SL(W)$. We choose $\underline{a}$ to be of the form $\underline{a} = (a_0m, m, \dots, m)$ for some sufficiently large integers $m > 0$ and $a_0 > 0$, where $m$ is chosen such that each $\LL_{\mathrm{Ch}_i}^m$ is very ample, and $a_0'$ is chosen according to the following result (where the reductive group $K$ is taken to be $R^{(1)} = SL(W)$):

\begin{lemma} \label{lemma: schmitt product stability}
	Let $K$ be a reductive group acting on a product of projective schemes $\prod_{i=0}^n Y_i$, where each $Y_i$ is endowed with a very ample $K$-linearisation $\LL_i$. Then there exists a positive integer $a' > 0$ such that for all $a_0 > a'$, for every point $p = (p_0, \dots, p_n) \in \prod_{i=0}^n Y_i$, if $p_0$ is $K$-stable with respect to the linearisation $\LL_0$ then $p$ is $K$-stable with respect to the linearisation $\LL_0^{a_0} \boxtimes \LL_1 \boxtimes \cdots \boxtimes \LL_n$.
\end{lemma}

\begin{proof}
	By using each $\LL_i$ to embed $Y_i$ inside a projective space, we may assume each $Y_i$ is a projective space $\PP(W_i)$. The lemma now easily follows from \cite[Proposition 1.7.3.1]{schmitt}.
\end{proof}	

\subsection{Upstairs Stabilisers and Weights}

Fix integers $\beta_1 > 0 > \beta_2 > \cdots > \beta_{n+1}$ with $\beta_1 \dim W + \sum_{i=2}^{n+1} \beta_i = 0$.\footnote{For the purposes of this paper, the specific choices of the $\beta_i$'s doesn't matter; what matters is that they satisfy the given constraints.} Define a 1PS $\lambda : \G_m \to SL(V)$ by setting
$$ \lambda(t) \cdot v = \begin{cases} t^{\beta_1} v & \text{if } v \in W, \\
	t^{\beta_{i+1}} v & \text{if } v = v_i. \end{cases} $$
Note that $P = P(\lambda)$ and that $\lambda$ grades the unipotent radical of the parabolic $P$. With this $\lambda$, define the 1PS $\lambda^{[1]}, \dots, \lambda^{[n]}$ etc. as in Section \ref{section: parabolic subgroups of SL(V)}. The 1PS $\lambda^{[i]}$ has weight space decomposition $V = Z^{i-1} \oplus V^{i-1}$ where $V^j = \bigoplus_{k > j} \C v_k$; vectors in $Z^{i-1}$ have $\lambda^{[i]}$-weight $\beta_{\leq i}$ and vectors in $V^{i-1}$ have $\lambda^{[i]}$-weight $\beta_{>i}$.

Fix a point $p$ of $\mathcal{S}_0'$ corresponding to a hyperplanar admissible flag $(X^0, X^1, \dots, X^n)$ with $X^i = X^n \cap \PP(Z^i)$ and $\langle X^i \rangle = \PP(Z^i) = \PP \left( W \oplus \bigoplus_{j \leq i} \C v_j \right)$ for all $i$. We introduce the following notation: for all $1 \leq i \leq j \leq n$, let $V^{ij} = \bigoplus_{k=i}^j \C v_k$. We also set $p^{[i]}(x) := \lim_{t \to 0} \lambda^{[i]}(t) \cdot x$ for a point $x \in \mathcal{H}$.

\begin{lemma} \label{lemma: limits of nice configurations}
	The following statements hold:
	\begin{enumerate}
		\item For each $i = 1, \dots, n$, the point $p^{[i]}(p) \in \mathcal{H}$ is given by the flag
		$$ p^{[i]}(p) = (X^0, \dots, X^{i-1}, J(X^{i-1}, \PP(V^{ii})), J(X^{i-1}, \PP(V^{i,i+1})), \dots, J(X^{i-1}, \PP(V^{in}))). $$
		\item For all $j \leq i$ we have $p^{[j]}(p^{[i]}(p)) = p^{[j]}(p)$.
	\end{enumerate}
\end{lemma}

\begin{proof}
	$\lambda^{[i]}$ fixes pointwise the subschemes $X^0, \dots, X^{i-1}$. On the other hand, as each $X^j \subset \PP(Z^j)$ is non-degenerate, for each $j \geq i$ the projection $X^j \dashrightarrow \PP(V^{ij})$ is dominant. The first assertion now follows from Lemma \ref{lemma: flat limits are joins}. For the second assertion, note that $Z^{j-1} \cap V^{ik} = 0$ for all $k \geq i$, so
	$$ J(X^{i-1}, \PP(V^{ik})) \cap \PP(Z^{j-1}) = X^{i-1} \cap \PP(Z^{j-1}) = X^{j-1}. $$
	The second assertion then follows from a second application of Lemma \ref{lemma: flat limits are joins}.
\end{proof}

\begin{lemma} \label{lemma: stabilisers of limits}
	For all $i = 1, \dots, n$, we have $\stab_{U^{[i]}}(p^{[i]}(p)) = \{e\}$.
\end{lemma}

\begin{proof}
	If $k \geq i$, then any projective automorphism of $\PP(V)$ arising from the unipotent group $U^{[i]}$ which preserves $J(X^{i-1}, \PP(V^{i,k}))$ must preserve $\PP(V^{i,k})$, as seen by considering dimensions of projective tangent spaces; points of $\PP(V^{i,k})$ have projective tangent space equal to all of $\PP(V)$, whereas points of $J(X^{i-1}, \PP(V^{i,k}))$ off $\PP(V^{i,k})$ have projective tangent spaces of dimension $(i-1) + (k - i) + 1 = k < \dim \PP(V)$. 
	
	It follows that any element $u = (A_{pq})_{p,q=1}^{n+1} \in \stab_{U^{[i]}}(p^{[i]}(p))$ must preserve each of $\PP(V^{i,i}), \dots, \PP(V^{i,n})$, which implies that if $p \neq q$ then $A_{pq} = 0$ whenever $p \leq i$ and $q > i$. On the other hand $u \in U^{[i]}$, so for $p \neq q$ we have $A_{pq} = 0$ if $q \leq i$ or if $p > i$. Consequently $u = e$ must be the identity matrix.
\end{proof}

\begin{lemma} \label{lemma: weights of limits}
	For each $i = 1, \dots, n$, the Hilbert--Mumford weight $\mu^{\LL_{\mathrm{Ch}}^{\underline{a}}}(\Psi(p), \lambda^{[i]})$ is the same for all points $p \in \mathcal{S}_0'$.
\end{lemma}

\begin{proof}
	First of all, since $d > \dim V - n = \dim W$, since $\beta_2, \dots, \beta_{n+1} < 0$ and since $d \not \in \left\{ \frac{\dim W + i - 1}{n + 1 - i} : i = 1, \dots, n \right\}$ then for each $i = 1, \dots, n$ we have
	$$ \beta_{>i} + d \beta_{\leq i} = \left( \frac{1}{n+1-i} - \frac{d}{\dim W + i - 1} \right)(\beta_{i+1} + \cdots + \beta_{n+1}) \neq 0. $$
	Therefore Lemma \ref{lemma: chow weights of joins} is applicable. 
	
	Choose a basis $y_1, \dots, y_M$ for $W^{\vee}$ and let $x_i \in V^{\vee}$ be dual to $v_i$. Write $\lambda^{[i]}(t) = \mathrm{diag}(t^{\beta_{\leq i}} I_{m_{\leq i}}, t^{\beta_{>i}} I_{m_{>i}})$ as in Section \ref{section: parabolic subgroups of SL(V)}. For a point $p \in \mathcal{S}_0'$, up to a positive scalar multiple we have
	\begin{multline*}
		\mu^{\LL_{\mathrm{Ch}}^{\underline{a}}}(\Psi(p), \lambda^{[i]}) = a_0 \mu^{\LL_{\mathrm{Ch}_1}}(\Psi_0(X^0), \lambda^{[i]}) \\ + \sum_{j=1}^{i-1} \mu^{\LL_{\mathrm{Ch}_j}}(\Psi_j(X^j), \lambda^{[i]}) + \sum_{j=i}^{n} \mu^{\LL_{\mathrm{Ch}_j}}(\Psi_j(J(X^{i-1}, \PP(V^{ij}))), \lambda^{[i]}).
	\end{multline*}
	
	The contribution $\mu^{\LL_{\mathrm{Ch}_j}}(\Psi_j(X^j), \lambda^{[i]})$ coming from $X^j$ (where $j < i$) is equal to $\beta_{\leq i}d(j+1)$ by Lemma \ref{lemma: chow weight with only one weight}. For the contribution coming from the join $J_j = J(X^{i-1}, \PP(V^{ij}))$ (where $j \geq i$), we have that the homogeneous ideal of $J_j \subset \PP(V)$ is given by
	$$ \mathbb{I}_{\PP(V)}(J_j) = \langle \mathbb{I}_{\PP(Z_j)}(J_j), x_{j+1}, \dots, x_n \rangle, $$
	so the homogeneous coordinate ring of $J_j \subset \PP(V)$ is the same as that for $J_j \subset \PP(Z^j)$. Applying Lemma \ref{lemma: chow weights of joins} then gives
	
	$$ \mu^{\LL_{\mathrm{Ch}_j}}(\Psi_j(J_j), \lambda^{[i]}) = \begin{cases} \beta_{>i}(j-i+1) & \text{if } j > 2i-1, \\ \beta_{\leq i} di & \text{if } j < 2i-1, \\ i(\beta_{>i} + d \beta_{\leq i}) & \text{if } j = 2i-1. \end{cases} $$
	
	It is now clear that the resulting expression for $\mu^{\LL_{\mathrm{Ch}}^{\underline{a}}}(\Psi(p), \lambda^{[i]})$ is the same for all points $p \in \mathcal{S}_0'$.
\end{proof}

\subsection{The Quotienting-in-Stages Procedure}

We now construct a quotient of $\mathcal{S}_0'$ in stages, following \cite[Construction 4.20]{hoskinsjackson}. After stage $i$, we will have quotiented by the group $\hat{H}^{(i)} = H^{(i)} \rtimes T^{(i)}$, and we pass from stage $i-1$ to stage $i$ by quotienting out the action of the group $\hat{H}_i = H_i \rtimes \lambda_i(\G_m)$. After $n$ stages, we will obtain a quotient for the action of $\hat{H}^{(n)} = P$, the parabolic group in Section \ref{section: reduction to P}.

\begin{construction} \label{construction: quotienting in stages}
	\emph{Base step:} As $\Psi$ restricts to an isomorphism $\mathcal{S}_0' \stackrel{\simeq}{\to} \Psi(\mathcal{S}_0')$, we can regard $\mathcal{S}_0' \subset \mathrm{Ch}$. Let $\mathcal{Q}_1$ be the projective scheme obtained by taking the closure of $\mathcal{S}_0'$ in $\mathrm{Ch}$. For some $s_0 > 0$, the $GL(V)$-linearisation $\LL_{\mathrm{Ch}}^{s_0\underline{a}}$ is very ample; let $\L_1$ be the restriction of this linearisation to $\mathcal{Q}_1$, twisted by a suitable character $\epsilon_1 \chi_1$ (to ensure that the corresponding $\hat{H}_1$-linearisation is adapted).
	
	Assuming we can apply Theorem \ref{theorem: U-hat in our setting}, there is a geometric $\hat{H}_1$-quotient $q_1 : (\mathcal{Q}_1)_{\min}^{\hat{H}_1-vs} \to (\mathcal{Q}_1)_{\min}^{\hat{H}_1-vs}/\hat{H}_1$ and a locally closed immersion $(\mathcal{Q}_1)_{\min}^{\hat{H}_1-vs}/\hat{H}_1 \subset \PP((H^0(\mathcal{Q}_1, \L_1^{s_1})^{\hat{H}_1})^{\vee})$ for some integer $s_1 > 0$, with the quotient being induced by the projection $\PP(H^0(\mathcal{Q}_1, \L_1^{s_1})^{\vee}) \dashrightarrow \PP((H^0(\mathcal{Q}_1, \L_1^{s_1})^{\hat{H}_1})^{\vee})$. Let $\mathcal{Q}_2$ be the closure of $(\mathcal{Q}_1)_{\min}^{\hat{H}_1-vs}/\hat{H}_1$ in $\PP((H^0(\mathcal{Q}_1, \L_1^{s_1})^{\hat{H}_1})^{\vee})$ and let $\L_2$ be the restriction of the $\OO(1)$ of $\PP((H^0(\mathcal{Q}_1, \L_1^{s_1})^{\hat{H}_1})^{\vee})$ to $\mathcal{Q}_2$, twisted by a suitable character $\epsilon_2 \chi_2$ to ensure adaptedness; then $\L_2$ is a very ample linearisation for a residual action of $P/\hat{H}^{(1)}$.
	
	\emph{Induction step:} Suppose $i \leq n-1$, and suppose we have constructed successive quotients $q_j$ of the form
	\[\begin{tikzcd}[ampersand replacement=\&,column sep=large]
		{(\mathcal{Q}_j)_{\min}^{\hat{H}_j-vs}} \&\& {(\mathcal{Q}_j)_{\min}^{\hat{H}_j-vs}/\hat{H}_j} \\
		\\
		{\mathcal{Q}_j} \&\& {\mathcal{Q}_{j+1}} \\
		\\
		{\PP(H^0(\mathcal{Q}_j, \L_j^{s_j})^{\vee})} \&\& {\PP((H^0(\mathcal{Q}_j, \L_j^{s_j})^{\hat{H}_j})^{\vee})}
		\arrow["{q_j}", from=1-1, to=1-3]
		\arrow["\circ"{description}, hook, from=1-1, to=3-1]
		\arrow["\circ"{description}, hook, from=1-3, to=3-3]
		\arrow["\shortmid"{marking}, hook, from=3-1, to=5-1]
		\arrow["\shortmid"{marking}, hook, from=3-3, to=5-3]
		\arrow["{q_j}", dashed, from=3-1, to=3-3]
		\arrow["{H^0(\mathcal{Q}_j, \L_j^{s_j})^{\hat{H}_j} \subset H^0(\mathcal{Q}_j, \L_j^{s_j})}"', dashed, from=5-1, to=5-3]
	\end{tikzcd}\]
	for each $j = 1, \dots, i$, where the top horizontal arrow is a geometric $\hat{H}_j$-quotient, where $\mathcal{Q}_{j+1}$ is the closure of the locally closed subscheme $(\mathcal{Q}_j)_{\min}^{\hat{H}_j-vs}/\hat{H}_j \subset \PP((H^0(\mathcal{Q}_j, \L_j^{s_j})^{\hat{H}_j})^{\vee})$, and where each $\mathcal{Q}_j$ is endowed with the very ample linearisation $\L_j$ obtained by restricting the $\OO(1)$ of $\PP((H^0(\mathcal{Q}_{j-1}, \L_{j-1}^{s_{j-1}})^{\hat{H}_{j-1}})^{\vee})$ to $\mathcal{Q}_j$ and then twisting by a suitable character $\epsilon_j \chi_j$ of $\hat{H}_j$ to ensure adaptedness. 
	
	The very ample invertible sheaf $\L_{i+1}$ on the projective scheme $\mathcal{Q}_{i+1}$ carries a linearisation for a residual action of $P/\hat{H}^{(i)}$, and in particular for $\hat{H}_{i+1}$. By twisting $\L_{i+1}$ by a character of $\hat{H}_{i+1}$ of the form $\epsilon_{i+1} \chi_{i+1}$, we can ensure that this $\hat{H}_{i+1}$-linearisation is adapted. Assuming once again that we can apply Theorem \ref{theorem: U-hat in our setting}, we obtain a diagram (for some integer $s_{i+1} > 0$) of the form
	\[\begin{tikzcd}[ampersand replacement=\&,column sep=large]
		{(\mathcal{Q}_{i+1})_{\min}^{\hat{H}_{i+1}-vs}} \&\& {(\mathcal{Q}_{i+1})_{\min}^{\hat{H}_{i+1}-vs}/\hat{H}_{i+1}} \\
		\\
		{\mathcal{Q}_{i+1}} \&\& {\mathcal{Q}_{i+2}} \\
		\\
		{\PP(H^0(\mathcal{Q}_{i+1}, \L_{i+1}^{s_{i+1}})^{\vee})} \&\& {\PP((H^0(\mathcal{Q}_{i+1}, \L_{i+1}^{s_{i+1}})^{\hat{H}_{i+1}})^{\vee})}
		\arrow["{q_{i+1}}", from=1-1, to=1-3]
		\arrow["\circ"{description}, hook, from=1-1, to=3-1]
		\arrow["\circ"{description}, hook, from=1-3, to=3-3]
		\arrow["\shortmid"{marking}, hook, from=3-1, to=5-1]
		\arrow["\shortmid"{marking}, hook, from=3-3, to=5-3]
		\arrow["{q_{i+1}}", dashed, from=3-1, to=3-3]
		\arrow["{H^0(\mathcal{Q}_{i+1}, \L_{i+1}^{s_{i+1}})^{\hat{H}_{i+1}} \subset H^0(\mathcal{Q}_{i+1}, \L_{i+1}^{s_{i+1}})}"', dashed, from=5-1, to=5-3]
	\end{tikzcd}\]
where the top horizontal arrow is a geometric $\hat{H}_{i+1}$-quotient and where $\mathcal{Q}_{i+2}$ is the closure of the locally closed subscheme $(\mathcal{Q}_{i+1})_{\min}^{\hat{H}_{i+1}-vs}/\hat{H}_{i+1} \subset \PP((H^0(\mathcal{Q}_{i+1}, \L_{i+1}^{s_{i+1}})^{\hat{H}_{i+1}})^{\vee})$. By defining $\L_{i+2}$ to be the restriction of the $\OO(1)$ of $\PP((H^0(\mathcal{Q}_{i+1}, \L_{i+1}^{s_{i+1}})^{\hat{H}_{i+1}})^{\vee})$ to $\mathcal{Q}_{i+2}$, twisted by a suitable character $\epsilon_{i+2} \chi_{i+2}$ to ensure adaptedness, the induction step of Construction \ref{construction: quotienting in stages} is complete. $\blacksquare$
\end{construction}

Assuming that Theorem \ref{theorem: U-hat in our setting} can always be applied at each stage, after $n$ stages Construction \ref{construction: quotienting in stages} terminates, yielding a diagram of the form
\[\begin{tikzcd}[ampersand replacement=\&]
	{\mathcal{Q}_1} \&\& {\mathcal{Q}_2} \&\& \cdots \&\& {\mathcal{Q}_n} \&\& {\mathcal{Q}_{n+1}} \\
	\\
	{(\mathcal{Q}_1)_{\min}^{\hat{H}_1-vs}} \&\& {(\mathcal{Q}_2)_{\min}^{\hat{H}_2-vs}} \&\& \cdots \&\& {(\mathcal{Q}_n)_{\min}^{\hat{H}_n-vs}}
	\arrow["\circ"{description}, hook', from=3-1, to=1-1]
	\arrow["\circ"{description}, hook', from=3-3, to=1-3]
	\arrow["\circ"{description}, hook', from=3-7, to=1-7]
	\arrow["{q_1}", dashed, from=1-1, to=1-3]
	\arrow["{q_2}", dashed, from=1-3, to=1-5]
	\arrow["{q_{n-1}}", dashed, from=1-5, to=1-7]
	\arrow["{q_n}", dashed, from=1-7, to=1-9]
	\arrow["{q_n}"', from=3-7, to=1-9]
	\arrow["{q_2}"', from=3-3, to=1-5]
	\arrow["{q_1}"', from=3-1, to=1-3]
\end{tikzcd}\]
where for each $i$ the restriction $q_i : (\mathcal{Q}_i)_{\min}^{\hat{H}_i-vs} \to q_i((\mathcal{Q}_i)_{\min}^{\hat{H}_i-vs}) = (\mathcal{Q}_i)_{\min}^{\hat{H}_i-vs}/\hat{H}_i$ is a geometric $\hat{H}_i$-quotient. If one sets for each $i = 1, \dots, n$
$$ q_{(i)} := q_i \circ q_{i-1} \circ \cdots \circ q_1 : \mathcal{Q}_1 \dashrightarrow \mathcal{Q}_{i+1}, $$
and if one inductively defines open subschemes $\mathcal{Q}_{(i)} \subset \mathcal{Q}_1$ for $i = 1, \dots, n$ by setting $\mathcal{Q}_{(1)} := (\mathcal{Q}_1)_{\min}^{\hat{H}_1-vs}$ and setting for $i > 1$
$$ \mathcal{Q}_{(i)} := \mathcal{Q}_{(i-1)} \cap q_{(i-1)}^{-1}( (\mathcal{Q}_i)_{\min}^{\hat{H}_i-vs} ), $$
then for each $i$, the morphism $q_{(i)} : \mathcal{Q}_{(i)} \to q_{(i)}(\mathcal{Q}_{(i)}) = \mathcal{Q}_{(i)}/\hat{H}^{(i)} \subset \mathcal{Q}_{i+1}$ is a well-defined geometric $\hat{H}^{(i)}$-quotient.

\subsection{The Parameter Space lies in the Domain of the Quotient}

Suppose we can show that Construction \ref{construction: quotienting in stages} can be carried out in full, so that there is a geometric $P$-quotient $\mathcal{Q}_{(n)}/P$. In order to show $\mathcal{S}_0'$ admits a geometric $P$-quotient, it is enough to prove that $\mathcal{S}_0'$ is an open subscheme of $\mathcal{Q}_{(n)}$. We will prove this by using an inductive argument to show that $\mathcal{S}_0'$ is an open subscheme of $\mathcal{Q}_{(i)}$ for each $i = 1, \dots, n$.
	
Fix a point $p$ of $\mathcal{S}_0'$ corresponding to a hyperplanar admissible flag flag $(X^0, X^1, \dots, X^n)$ with $X^i = X^n \cap \PP(Z^i)$ and $\langle X^i \rangle = \PP(Z^i) = \PP \left( W \oplus \bigoplus_{j \leq i} \C v_j \right)$ for all $i$. Let $x = \Psi(p) \in \mathcal{Q}_1$.

We introduce the following notation:
\begin{enumerate}[label=(\roman*)]
    \item As above, set $\mathcal{Q}_{(1)} := (\mathcal{Q}_1)_{\min}^{\hat{H}_1-vs}$.
    \item For all $i = 1, \dots, n$, set $\mathcal{Q}_1^{[i]} := \mathcal{Q}_1(\lambda^{[i]})_{\min}$ and set $Z_1^{[i]} := Z(\mathcal{Q}_1, \lambda^{[i]})_{\min}$. 
    \item Denote $p^{[i]}$ the $\lambda^{[i]}$-retraction $p^{[i]} : \mathcal{Q}_1^{[i]} \to Z_1^{[i]}$.
\end{enumerate}

\begin{lemma} \label{lemma: qis base step}
    The following statements hold:
    \begin{enumerate}
        \item the point $x$ lies in $\mathcal{Q}_{(1)}$, and for all $i \geq 1$ we have $x \in \mathcal{Q}_1^{[i]}$;
        \item for all $i > 1$, we have $p^{[i]}(x) \in \mathcal{Q}_{(1)}$; and
        \item for all $i \geq 1$, we have $p^{[i]}(x) \in Z_1^{[i]}$.
    \end{enumerate}
    In particular, for the action $\hat{H}_1 \circlearrowright \mathcal{Q}_1$ linearised by $L_1$, Theorem \ref{theorem: U-hat in our setting} is applicable.
\end{lemma}

\begin{proof}
    We begin by considering the third statement. From Lemma \ref{lemma: weights of limits}, the Hilbert--Mumford weight of $p^{[i]}(x)$ with respect to the linearisation $L_1$ and the 1PS $\lambda^{[i]}$ is constant across the dense open subscheme $\mathcal{S}_0' \equiv \Psi(\mathcal{S}_0') \subset \mathcal{Q}_1$. This common weight must be the minimal weight, and so $p^{[i]}(x) \in Z_1^{[i]}$ as claimed.
    
    Let us now show that the points $x, p^{[2]}(x), \dots, p^{[n]}(x)$ are in $\mathcal{Q}_{(1)} = (\mathcal{Q}_1)_{\min}^{\hat{H}_1-vs}$; this involves establishing the following for all elements $y \in \{x, p^{[2]}(x), \dots, p^{[n]}(x)\}$:
    \begin{enumerate}[label=(\roman*)]
        \item  $y \in \mathcal{Q}_1^{[1]} = \mathcal{Q}_1(\lambda^{[1]})_{\min}$, that is $p^{[1]}(y)$ has minimal $\lambda^{[1]}$-weight.
        \item $p^{[1]}(y)$ has trivial $U^{[1]}$-stabiliser and is stable under the action of $R_1 = R^{(1)} = SL(W)$ on $Z_1^{[1]}$.
        \item $y$ is not contained in the $U^{[1]}$-sweep of $Z_1^{[1]}$.
    \end{enumerate}
    However, by Lemma \ref{lemma: limits of nice configurations} we have for all such points $y$ that
    $$ p^{[1]}(y) = p^{[1]}(x) = \Psi(X^0, J(X^0, \PP(V^{1,1})), \dots, J(X^0, \PP(V^{1,n}))). $$
    
    We have already shown that this point has minimal $\lambda^{[1]}$-weight, so each such point $y$ is in $\mathcal{Q}_1^{[1]}$. By Lemma \ref{lemma: schmitt product stability}, the Chow stability of $X^0 \subset \PP(Z^0)$ together with Proposition \ref{prop: chow stability dim 0} implies that the point $p^{[1]}(x)$ is $SL(W)$-stable with respect to $\L_1$. Lemma \ref{lemma: stabilisers of limits} implies that $\stab_{U^{[1]}}(p_{[1]}(x))$ is trivial. Finally, points in $U^{[1]} Z_1^{[1]}$ are fixed by the retraction $p^{[1]}$, whereas no point $y \in \{x, p^{[2]}(x), \dots, p^{[n]}(x)\}$ is fixed by $p^{[1]}$, as the dimension 1 component of a flag corresponding to such a point $y$ (namely, $X^1$) does not coincide with that of $p^{[1]}(y)$. Consequently each such point $y$ cannot lie in $U^{[1]} Z_1^{[1]}$. This completes the proof of the Lemma.
\end{proof}

We now consider the induction step of Construction \ref{construction: quotienting in stages}. Suppose $1 \leq i < \ell - 1 = n$, and suppose we have constructed successive quotients $q_j$ of the form
\[\begin{tikzcd}[ampersand replacement=\&,column sep=6em,row sep=1em]
    {(\mathcal{Q}_j)_{\min}^{\hat{H}_j-vs}} \&\& {(\mathcal{Q}_j)_{\min}^{\hat{H}_j-vs}/\hat{H}_j} \\
    \\
    {\mathcal{Q}_j} \&\& {\mathcal{Q}_{j+1}} \\
    \\
    {\PP(H^0(\mathcal{Q}_j, \L_j^{s_j})^{\vee})} \&\& {\PP((H^0(\mathcal{Q}_j, \L_j^{s_j})^{\hat{H}_j})^{\vee})}
    \arrow["{q_j}", from=1-1, to=1-3]
    \arrow["\circ"{description}, hook, from=1-1, to=3-1]
    \arrow["\circ"{description}, hook, from=1-3, to=3-3]
    \arrow["\shortmid"{marking}, hook, from=3-1, to=5-1]
    \arrow["\shortmid"{marking}, hook, from=3-3, to=5-3]
    \arrow["{q_j}", dashed, from=3-1, to=3-3]
    \arrow["{H^0(\mathcal{Q}_j, \L_j^{s_j})^{\hat{H}_j} \subset H^0(\mathcal{Q}_j, \L_j^{s_j})}"', dashed, from=5-1, to=5-3]
\end{tikzcd}\]
for each $j = 1, \dots, i$, where the top horizontal arrow is a geometric $\hat{H}_j$-quotient, where $\mathcal{Q}_{j+1}$ is the closure of the locally closed subscheme
$$(\mathcal{Q}_j)_{\min}^{\hat{H}_j-vs}/\hat{H}_j \subset \PP((H^0(\mathcal{Q}_j, \L_j^{s_j})^{\hat{H}_j})^{\vee}),$$ 
and where each $\mathcal{Q}_j$ is endowed with the very ample linearisation $\L_j$ obtained by restricting the $\OO(1)$ of $\PP((H^0(\mathcal{Q}_{j-1}, \L_{j-1}^{s_{j-1}})^{\hat{H}_{j-1}})^{\vee})$ to $\mathcal{Q}_j$ and then twisting by a suitable character $\epsilon_j \chi_j$ of $\hat{H}_j$ to ensure adaptedness. 

As in Construction \ref{construction: quotienting in stages}, the very ample invertible sheaf $\L_{i+1}$ on the projective scheme $\mathcal{Q}_{i+1}$ carries a linearisation for a residual action of $P/\hat{H}^{(i)}$, and in particular for $\hat{H}_{i+1}$. By twisting $\L_{i+1}$ by a character of $\hat{H}_{i+1}$ of the form $\epsilon_{i+1} \chi_{i+1}$ if necessary, we may assume that this $\hat{H}_{i+1}$-linearisation is adapted.

For the induction step, we need to show that Theorem \ref{theorem: U-hat in our setting} can be applied to this linearised action of $\hat{H}_{i+1}$. In order to establish this, we introduce the following notation:

\begin{enumerate}[label=(\roman*)]
    \item For $j = 1, \dots, i$, set
    $$ q_{(j)} := q_j \circ q_{j-1} \circ \cdots \circ q_1 : \mathcal{Q}_1 \dashrightarrow \mathcal{Q}_{j+1}. $$
    \item We inductively define open subschemes $\mathcal{Q}_{(j)} \subset \mathcal{Q}_1$ for $j = 1, \dots, i+1$ by setting $\mathcal{Q}_{(1)} := (\mathcal{Q}_1)_{\min}^{\hat{H}_1-vs}$ (as before) and setting for $j > 1$
    $$ \mathcal{Q}_{(j)} := \mathcal{Q}_{(j-1)} \cap q_{(j-1)}^{-1}( (\mathcal{Q}_j)_{\min}^{\hat{H}_j-vs} ), $$
    so that for each $j$, the morphism $q_{(j)} : \mathcal{Q}_{(j)} \to q_{(j)}(\mathcal{Q}_{(j)}) = \mathcal{Q}_{(j)}/\hat{H}^{(j)} \subset \mathcal{Q}_{j+1}$ is a well-defined geometric $\hat{H}^{(j)}$-quotient.
    \item For all $k = 1, \dots, i+1$ and for all $k \leq j < \ell = n+1$, let $\mathcal{Q}_k^{[j]} = \mathcal{Q}_k(\lambda_k^{[j]})_{\min}$, let $Z_k^{[j]} = Z(\mathcal{Q}_k, \lambda_k^{[j]})_{\min}$, and let $p_k^{[j]} : \mathcal{Q}_k^{[j]} \to Z_k^{[j]}$ be the retraction under $\lambda_k^{[j]}$. As a special case of the above, we continue to denote $p^{[j]} = p_1^{[j]}$ for the $\lambda_1^{[j]} = \lambda^{[j]}$ retraction on $\mathcal{Q}_1$.
    \item Set $\mathcal{Q}_{(0)} := \mathcal{Q}_1$ and $q_{(0)} := \mathrm{id}_{\mathcal{Q}_1}$.
\end{enumerate}

\begin{lemma} \label{lemma: qis induction step}
    For each $j = 0, 1, \dots, i$, the following statements hold:
    \begin{enumerate}
        \item $x \in \mathcal{Q}_{(j+1)}$, and for all $k \geq j + 1$ we have $q_{(j)}(x) \in \mathcal{Q}_{j+1}^{[k]}$;
        \item for all $j + 1 < k \leq n$, we have $p^{[k]}(x) = p_1^{[k]}(x) \in \mathcal{Q}_{(j+1)}$; and 
        \item for all $j+1 \leq k \leq n$, $q_{(j)}(p^{[k]}(x)) \in Z_{j+1}^{[k]}$.
    \end{enumerate}
\end{lemma}

\begin{proof}
    We argue by induction on $j$, noting that the base case is Lemma \ref{lemma: qis base step}. Assume that the assertions of Lemma \ref{lemma: qis induction step} holds for each $j = 0, \dots, r-1$, where $r \leq i$. We first show that for each $k \geq r+1$, the point $q_{(r)}(p^{[k]}(x)) = p_{r+1}^{[k]}(q_{(r)}(x))$ is contained in $Z_{r+1}^{[k]}$. For each $k \geq r+1$, the induction hypothesis implies that $q_{(r-1)}(p^{[k]}(x)) = p_r^{[k]}(q_{(r-1)}(x)) \in Z_{r}^{[k]} \cap (\mathcal{Q}_r)_{\min}^{\hat{H}_r-vs}$. In particular, we have $Z_{r}^{[k]} \cap (\mathcal{Q}_r)_{\min}^{\hat{H}_r-vs} \neq \emptyset$. From Construction \ref{construction: quotienting in stages} the map $q_r : \mathcal{Q}_r \dashrightarrow \mathcal{Q}_{r+1}$ is given by the projection corresponding to the inclusion $B := H^0(\mathcal{Q}_r, \L_r^{s_r})^{\hat{H}_r} \subset A := H^0(\mathcal{Q}_r, \L_r^{s_r})$, and the twist of $\L_{r+1}$ by the character $-\epsilon_{r+1} \chi_{r+1}$ pulls back under $q_r$ to give $\L_r^{s_r}$. As such, the arguments given in the proof of \cite[Lemma 5.4]{hoskinsjackson} carry over to the map $q_r : \mathcal{Q}_r \dashrightarrow \mathcal{Q}_{r+1}$. Namely, that $Z_{r}^{[k]} \cap (\mathcal{Q}_r)_{\min}^{\hat{H}_r-vs} \neq \emptyset$ implies that the maximal weights for $\lambda_r^{[k]}(\G_m)$ acting on both $A$ and $B$ coincide, so we may choose a basis $\mathcal{B}_A$ for $A$ of the form
    $$ \mathcal{B}_A = \{ \alpha_{\iota}, \alpha_{\kappa}', \beta_{\mu}, \beta_{\nu}' \}, $$
    where the $\alpha_{\iota}, \alpha_{\kappa}'$ are a basis for $B$ and the $\alpha_{\iota}, \beta_{\mu}$ are a basis for the maximal weight space for $\lambda_r^{[k]}(\G_m)$ in $A$. With respect to this basis, the morphism $q_r$ corresponds to projecting onto the $\alpha_{\iota}, \alpha_{\kappa}'$-coordinates, and the retraction $p_r^{[k]}$ on $\mathcal{Q}_r$ corresponds to projecting onto the $\alpha_{\iota}, \beta_{\mu}$-coordinates. Setting $z := q_{(r-1)}(p^{[k]}(x))$, since $z \in Z_r^{[k]}$ then all $\alpha_{\kappa}', \beta_{\nu}'$-coordinates of $z$ vanish, and since $z \in (\mathcal{Q}_r)_{\min}^{\hat{H}_r-vs}$ then at least one $\alpha_{\iota}$-coordinate of $z$ does not vanish. Applying $q_r$ to $z$ (that is, projecting onto the $\alpha_{\iota}$-coordinates), we obtain that the point $q_r(z) = q_{(r)}(p^{[k]}(x)) \in \mathcal{Q}_{r+1}$ is fixed by $\lambda_{r+1}^{[k]}$ and has all coordinates in the maximal weight space, whence $z \in Z_{r+1}^{[k]}$. Consequently, we find that for all $k \geq r+1$,
    $$ p_{r+1}^{[k]}(q_{(r)}(x)) = q_{(r)}(p^{[k]}(x)) = q_r(q_{(r-1)}(p^{[k]}(x))) \in Z_{r+1}^{[k]}. $$
    In turn this implies that $q_{(r)}(x) \in \mathcal{Q}_{r+1}^{[k]}$ for all $k \geq r+1$.
    
    Let us now show that the points $x, p^{[r+2]}(x), \dots, p^{[n]}(x)$ lie in $\mathcal{Q}_{(r+1)}$, that is their images under $q_{(r)}$ lie in $(\mathcal{Q}_{r+1})_{\min}^{\hat{H}_{r+1}-vs}$. Note that the semisimple part $R_{r+1}$ of the Levi subgroup $R_{r+1}$ of $\hat{H}_{r+1}$ is trivial, so $R_{r+1}$-stability is vacuous. As such, it suffices to establish the following for any point $y \in \{q_{(r)}(x), q_{(r)}(p^{[r+2]}(x)), \dots, q_{(r)}(p^{[n]}(x))\}$:
    \begin{enumerate}[label=(\roman*)]
        \item $p_{r+1}^{[r+1]}(y) \in Z_{r+1}^{[r+1]}$.
        \item $p_{r+1}^{[r+1]}(y)$ has trivial $U_{r+1}$-stabiliser.
        \item $y$ is not contained in the $U_{r+1}$-sweep of $Z_{r+1}^{[r+1]}$.
    \end{enumerate}
    By Lemma \ref{lemma: limits of nice configurations} we have for all such points $y$
    $$ p_{r+1}^{[r+1]}(y) = p_{r+1}^{[r+1]}(q_{(r)}(x)) = q_{(r)}(p^{[r+1]}(x)), $$
    which is represented by the $\hat{H}^{(r)}$-orbit of the configuration
    $$ p^{[r+1]}(p) = (X^0, \dots, X^r, J(X^r, \PP(V^{r+1, r+1})), \dots, J(X^r, \PP(V^{r+1,n}))).$$
    of subschemes of $\PP(V)$. We have already established that $q_{(r)}(p^{[r+1]}(x))$ is in $Z_{r+1}^{[r+1]}$. Combining Lemma \ref{lemma: stabilisers of limits} with \cite[Lemma 5.12]{hoskinsjackson}, the triviality of $\stab_{U^{[r+1]}}(p^{[r+1]}(p))$ implies that $q_{(r)}(p^{[r+1]}(x))$ has trivial $U_{r+1}$-stabiliser. Any point in $q_{(r)}(\mathcal{Q}_{(r)}) \cap U_{r+1} Z_{r+1}^{[r+1]}$ is fixed under the retraction $p_{r+1}^{[r+1]}$. However none of the points $y$ are fixed by $p_{r+1}^{[r+1]}$, as they are all represented by configurations whose dimension $(r+1)$ component is the non-singular variety $X^{r+1}$, whereas $p_{r+1}^{[r+1]}(y)$ is represented in dimension $(r+1)$ by the singular variety $J(X^r, \PP(V^{r+1,r+1}))$. 
    
    Consequently none of these points $y$ can lie in $U_{r+1} Z_{r+1}^{[r+1]}$. It follows that each $y \in \{q_{(r)}(x), q_{(r)}(p^{[r+2]}(x)), \dots, q_{(r)}(p^{[n]}(x))\}$ lies in $(\mathcal{Q}_{r+1})_{\min}^{\hat{H}_{r+1}-vs}$. This concludes the induction step and in turn concludes the proof of the lemma.
\end{proof}

\begin{cor}
    For the action $\hat{H}_{i+1} \circlearrowright \mathcal{Q}_{i+1}$ linearised by $\L_{i+1}$, Theorem \ref{theorem: U-hat in our setting} is applicable.
\end{cor}

\begin{proof}
    This immediately follows from Lemma \ref{lemma: qis induction step}.
\end{proof}

Arguing by induction on $i$, it follows that Construction \ref{construction: quotienting in stages} can be carried out in full, yielding a diagram of the form
\[\begin{tikzcd}[ampersand replacement=\&,cramped]
    {\mathcal{Q}_1} \&\& {\mathcal{Q}_2} \&\& \cdots \&\& {\mathcal{Q}_n} \&\& {\mathcal{Q}_{n+1}} \\
    \\
    {(\mathcal{Q}_1)_{\min}^{\hat{H}_1-vs}} \&\& {(\mathcal{Q}_2)_{\min}^{\hat{H}_2-vs}} \&\& \cdots \&\& {(\mathcal{Q}_n)_{\min}^{\hat{H}_n-vs}}
    \arrow["\circ"{description}, hook', from=3-1, to=1-1]
    \arrow["\circ"{description}, hook', from=3-3, to=1-3]
    \arrow["\circ"{description}, hook', from=3-7, to=1-7]
    \arrow["{q_1}", dashed, from=1-1, to=1-3]
    \arrow["{q_2}", dashed, from=1-3, to=1-5]
    \arrow["{q_{n-1}}", dashed, from=1-5, to=1-7]
    \arrow["{q_n}", dashed, from=1-7, to=1-9]
    \arrow["{q_n}"', from=3-7, to=1-9]
    \arrow["{q_2}"', from=3-3, to=1-5]
    \arrow["{q_1}"', from=3-1, to=1-3]
\end{tikzcd}\]
where for each $i$ the restriction $q_i : (\mathcal{Q}_i)_{\min}^{\hat{H}_i-vs} \to q_i((\mathcal{Q}_i)_{\min}^{\hat{H}_i-vs}) = (\mathcal{Q}_i)_{\min}^{\hat{H}_i-vs}/\hat{H}_i$ is a geometric $\hat{H}_i$-quotient. With
$$ q_{(n)} := q_n \circ q_{n-1} \circ \cdots \circ q_1 : \mathcal{Q}_1 \dashrightarrow \mathcal{Q}_{n+1} $$
and
$$ \mathcal{Q}_{(n)} := \mathcal{Q}_{(n-1)} \cap q_{(n-1)}^{-1}((\mathcal{Q}_n)_{\min}^{\hat{H}_n-vs}), $$
the morphism $q_{(n)} : \mathcal{Q}_{(n)} \to q_{(n)}(\mathcal{Q}_{(n)}) \subset \mathcal{Q}_{n+1}$ is a well-defined geometric $\hat{H}^{(n)}$-quotient of $\mathcal{Q}_{(n)}$. If $x$ is a point of $\mathcal{S}_0' \equiv \Psi(\mathcal{S}_0')$, by iterating Lemma \ref{lemma: qis induction step} one has $x \in \mathcal{Q}_{(n)}$.

\subsection{Completing the Proof}

By tying everything together, we complete the proof of Theorem \ref{theorem: main theorem}.

\begin{proof}[Proof of Theorem \ref{theorem: main theorem}]
	By Corollary \ref{cor: parameter space to be quotiented}, the moduli functor $\mathcal{F}_{n,d,\underline{\Phi}}^{\PP(V)}$ admits a coarse moduli space if and only if the scheme $\mathcal{S}'$ admits a categorical quotient for the action of $SL(V)$ which is an orbit space morphism. From Section \ref{section: reduction to P}, this is equivalent to $\mathcal{S}_0' \equiv \Psi(\mathcal{S}_0')$ admitting a categorical quotient for the action of $P$ which is an orbit space morphism. 
			
    The scheme $\mathcal{Q}_{(n)}$ admits a quasi-projective geometric $\widehat{H}^{(n)}$-quotient. However, using the notation of Section \ref{section: parabolic subgroups of SL(V)} there are surjections $R \to L/T$ and $H \to P/T$, both with finite kernels. The residual actions of these finite kernals are not a problem, since the action of a finite group on a quasi-projective scheme always admits a quasi-projective geometric quotient (see for instance \cite[Paragraph after Remark 2.1.1.1]{schmitt}). As such, the scheme $\mathcal{Q}_{(n)}$ admits a quasi-projective geometric $P$-quotient.
    
    There is an inclusion $\Psi(\mathcal{S}_0') \subset \mathcal{Q}_{(n)}$ of open subschemes of $\mathcal{Q}_{(1)}$. A geometric $P$-quotient of $\mathcal{S}_0'$ can then be obtained by restricting the geometric $P$-quotient $q_{(n)} : \mathcal{Q}_{(n)} \to q_{(n)}(\mathcal{Q}_{(n)}) \subset \mathcal{Q}_{n+1}$, by applying Lemma \ref{lemma: restricting geometric quotients locally trivial} to the $U_i$-quotients (which are all Zariski-locally trivial) and Lemma \ref{lemma: restricting geometric quotients reductive} to the quotients by the $\lambda_i(\G_m)$ and by $R_1 = SL(W)$ (which all have the property of being geometric quotients). In particular, $\mathcal{S}_0' \equiv \Psi(\mathcal{S}_0')$ admits a categorical quotient for the action of $P$ which is an orbit space morphism. This completes the proof of Theorem \ref{theorem: main theorem}.
\end{proof}

\section{Extending the Main Result} \label{section: extending the main result}

We conclude by briefly considering a couple of ways in which Theorem \ref{theorem: main theorem} can be extended.

\subsection{Weighting the Points}

The construction presented above treats $X^0$ as an unordered collection of points in the linear subspace $Z^0 \subset \PP(V)$. It is possible to consider a modified setup, where we instead consider $X^0$ to be an \emph{ordered} collection of $d$ points in $Z^0$, by replacing the Hilbert scheme $\mathrm{Hilb}(\PP(V), d)$ of $0$-dimensional length $d$ subschemes of $\PP(V)$ with the $d$-fold product $\PP(V)^{\times d}$. In place of the Chow linearisation $\LL_{\mathrm{Ch}_0}$, one may take the linearisation
$$ \OO_{\PP(V)}(\underline{w}) := \boxtimes_{i=1}^d \OO_{\PP(V)}(w_i), \quad \underline{w} = (w_1, \dots, w_d) \in (\Z^{>0})^d. $$

With respect to the induced action of $SL(W)$ arising from a given decomposition $V = W \oplus W'$, it follows from \cite[Theorem 11.1]{dolgachevbook} that if $p = (p_1, \dots, p_d) \in \PP(W)^{\times n} \subset \PP(V)^{\times n}$ then $p$ is $SL(W)$-(semi)stable with respect to the linearisation $\OO_{\PP(V)}(\underline{w})$ if and only if for all proper linear subspaces $Z \subset \PP(W)$,
\begin{equation} \label{eqn: weighted point inequality}
    \frac{\sum_{p_i \in Z} w_i}{\sum_{i=1}^d w_i} < (\leq) \ \frac{\dim Z + 1}{\dim W}
\end{equation}
Applying the same quotienting-in-stages construction used to prove Theorem \ref{theorem: main theorem} with this modified setup, one obtains (for each $\underline{w} \in (\Z^{\geq 0})^d$) a coarse moduli space parametrising all non-degenerate, non-singular and stable hyperplanar admissible flags $(\underline{X}, \underline{Z})$ together with a labelling of the points of $X^0$, where the appropriate notion of stability is given by requiring that the labelled points in $X^0 \subset \PP(Z^0)$ satisfy Inequality \ref{eqn: weighted point inequality} strictly.

\subsection{Omitting the Points} \label{section: omitting points}

Another variant of Theorem \ref{theorem: main theorem} can be obtained in which the flags being parametrised are of the form
\begin{equation} \label{eqn: flag starting from a curve}
    X^1 \subset \cdots \subset X^{n-1} \subset X^n, \quad Z^1 \subset \cdots \subset Z^{n-1} \subset Z^n = V
\end{equation}
where $X^1 \subset \PP(Z^1)$ is a smooth, non-degenerate connected projective curve which is now required to be GIT stable with respect to the Chow linearisation on $\mathrm{Hilb}(\PP(Z^1), \Phi_1)$ (and all other subvarieties $X^i \subset \PP(Z^i)$ are non-degenerate, smooth and connected). As above, the construction proceeds along very similar lines to the quotienting-in-stages construction used to prove Theorem \ref{theorem: main theorem}; the locus $\mathcal{S}'$ being quotiented is taken to be the analogously-defined locally closed subscheme of $\prod_{i=1}^n \mathcal{H}_i \times \prod_{j=1}^n \mathrm{Gr}_j$, and the group $P$ is taken to be the parabolic subgroup of $SL(V)$ preserving the flag $Z^1 \subset \cdots \subset Z^n = V$. In the base step of the quotienting-in-stages procedure (cf. Lemma \ref{lemma: qis base step}), the GIT stability of $X^1 \subset \PP(Z^1)$ is used to argue that the points $p^{[1]}(y)$, $y \in \{ x, p^{[2]}(x), \dots, p^{[n-1]}(x) \}$ are all stable with respect to the reductive linear algebraic group $R_1$; otherwise the quotienting-in-stages construction proceeds in the same manner as the preceding section.

From the results of \cite{bfmv}, it is known that if $g \geq 2$ and if $d > 2(2g-2)$ then smooth non-degenerate connected projective curves in $\PP^{d-g}$ of degree $d$ and genus $g$ are Chow stable; in particular, the Chow stable locus of the corresponding Hilbert scheme of curves is non-empty. Imposing the requirement that the degree and genus of $X^1$ satisfy these constraints imposes constraints on the higher-dimensional $X^i$ fitting into a flag of the form \eqref{eqn: flag starting from a curve}. Indeed, fixing $k \in \{2, \dots, n\}$, it follows from the Hirzebruch--Riemann--Roch theorem (as applied in \cite[Example 18.3.5]{fultonintersection}) that if $H \subset X^k$ is the hyperplane class then
\begin{align*}
    1 - g &= \chi(X^1, \OO_{X^1}) \\ 
    &= \int_{X^k} \left( \prod_{i=1}^{k-1} (1 - \exp(-H)) \cap \mathrm{Td}(X^k) \right) \\
    &= \int_{X^k} \left( H^{k-1} \cdot \left( -\frac{K_{X^k}}{2} \right) - \frac{k-1}{2} H^k \right) \\
    &= -\frac{k-1}{2} \cdot d - \frac{K_{X^k} \cdot X^1}{2}.
\end{align*}
Here $\mathrm{Td}(X^k) = 1 - \frac{1}{2} K_{X^k} + \cdots$ is the Todd class of $X^k$. The inequality $d > 2(2g-2)$ implies that if $g \geq 2$ then
$$ K_{X^k} \cdot X^1 < -(2k-3)(2g-2) \leq 0. $$
In particular, each canonical divisor $K_{X^k}$ cannot be nef or numerically trivial.

For essentially the same reasons as in the case $k = 1$, it is also possible to obtain a variant of Theorem \ref{theorem: main theorem} to construct for each $k \geq 1$ a coarse moduli space parametrising projective equivalence classes of flags of the form
$$ X^k \subset \cdots \subset X^{n-1} \subset X^n, \quad Z^k \subset \cdots \subset Z^{n-1} \subset Z^n = V $$
where this time $X^k \subset \PP(Z^k)$ is a smooth, non-degenerate, \emph{Chow stable} subvariety of $\PP(Z^k)$.

\bibliographystyle{acm}
\bibliography{admrefs.bib}

\begin{thebibliography}{10}

\bibitem{angeniol}
{\sc Ang{\'e}niol, B.}
\newblock {\em {Familles de Cycles Alg{\'e}briques -- Sch{\'e}ma de Chow}}, vol.~896 of {\em Lecture Notes in Mathematics}.
\newblock Springer, 1981.

\bibitem{geometryofalgcurvesvol1}
{\sc Arbarello, E., Cornalba, M., Griffiths, P., and Harris, J.}
\newblock {\em {Geometry of Algebraic Curves: Volume I}}.
\newblock Grundlehren der Mathematischen Wissenschaften (A Series of Comprehensive Studies in Mathematics). Springer, 1985.

\bibitem{aubin}
{\sc Aubin, T.}
\newblock {Équations du Type Monge--Ampère sur les Variétés K{\"a}hlériennes Compactes}.
\newblock {\em Bull. Sci. Math. 102}, 1 (1978), 63--95.

\bibitem{baldwinswinarski}
{\sc Baldwin, E., and Swinarski, D.}
\newblock {A Geometric Invariant Theory Construction of Moduli Spaces of Stable Maps}.
\newblock {\em Int. Math. Res. Not. 2008\/} (2008).

\bibitem{bdhkconstructing}
{\sc B{\'e}rczi, G., Doran, B., Hawes, T., and Kirwan, F.}
\newblock {Constructing Quotients of Algebraic Varieties by Linear Algebraic Group Actions}.
\newblock In {\em {Handbook of Group Actions, Volume IV}}, L.~Ji, A.~Papadopoulos, and S.-T. Yau, Eds., vol.~41 of {\em Advanced Lectures in Mathematics}. International Press of Boston, Inc., 2018, pp.~331--446.

\bibitem{bdhkapplications}
{\sc B{\'e}rczi, G., Doran, B., Hawes, T., and Kirwan, F.}
\newblock {Geometric Invariant Theory for Graded Unipotent Groups and Applications}.
\newblock {\em Journal of Topology 11}, 3 (2018), 826--855.

\bibitem{bkmomentmaps}
{\sc B{\'e}rczi, G., and Kirwan, F.}
\newblock {Moment Maps and Cohomology of Non-Reductive Quotients}.
\newblock {\em arXiv preprint arXiv:1909.11495\/} (2019).

\bibitem{bfmv}
{\sc Bini, G., Felici, F., Melo, M., and Viviani, F.}
\newblock {\em {Geometric Invariant Theory for Polarized Curves}}.
\newblock Lecture Notes in Mathematics 2122. Springer, 2014.

\bibitem{chowvdwaerden}
{\sc Chow, W.-L., and Van~der Waerden, B.}
\newblock {Zur Algebraischen Geometrie. IX: {\"U}ber Zugeordnete Formen und Algebraische Systeme von Algebraischen Mannigfaltigkeiten}.
\newblock {\em Mathematische Annalen 113}, 1 (1937), 692--704.

\bibitem{dolgachevbook}
{\sc Dolgachev, I.}
\newblock {\em {Lectures on Invariant Theory}}, vol.~296 of {\em London Mathematical Society Lecture Note Series}.
\newblock Cambridge University Press, 2003.

\bibitem{donaldson}
{\sc Donaldson, S.}
\newblock {Scalar Curvature and Projective Embeddings, I}.
\newblock {\em J. Diff. Geom. 59}, 3 (2001), 479--522.

\bibitem{fultonintersection}
{\sc Fulton, W.}
\newblock {\em Intersection Theory}, 3~ed.
\newblock Ergebnisse der Mathematik und ihrer Grenzgebiete. Springer, 2013.

\bibitem{giesekersurfaces}
{\sc Gieseker, D.}
\newblock {Global Moduli for Surfaces of General Type}.
\newblock {\em Inventiones Mathematicae 43\/} (1977), 233--282.

\bibitem{gieseker}
{\sc Gieseker, D.}
\newblock {\em {Lectures on Moduli of Curves}}.
\newblock Tata Institute Lectures on Mathematics and Physics. Springer-Verlag, 1982.

\bibitem{harrisbasicag}
{\sc Harris, J.}
\newblock {\em {Algebraic Geometry: A First Course}}, vol.~133 of {\em Graduate Texts in Mathematics}.
\newblock Springer Science \& Business Media, 1992.

\bibitem{hoskinsjackson}
{\sc Hoskins, V., and Jackson, J.}
\newblock {Quotients by Parabolic Groups and Moduli Spaces of Unstable Objects}.
\newblock {\em arXiv preprint arXiv:2111.07429\/} (2021).

\bibitem{lazarsfeldmustata}
{\sc Lazarsfeld, R., and Mustaț{\u{a}}, M.}
\newblock {Convex Bodies Associated to Linear Series}.
\newblock {\em Annales Scientifiques de l'{\'E}cole Normale Sup{\'e}rieure 42}, 5 (2009), 783--835.

\bibitem{milnegroups}
{\sc Milne, J.}
\newblock {\em {Algebraic Groups: The Theory of Group Schemes of Finite Type over a Field}}, vol.~170 of {\em Cambridge Studies in Advanced Mathematics}.
\newblock Cambridge University Press, 2017.

\bibitem{miranda1981moduli}
{\sc Miranda, R.}
\newblock {The Moduli of Weierstrass Fibrations over P1}.
\newblock {\em Math. Ann. 255\/} (1981), 379--394.

\bibitem{morrisonsurfaces}
{\sc Morrison, I.}
\newblock {Projective Stability of Ruled Surfaces}.
\newblock {\em Inv. Math. 56}, 3 (1980), 269--304.

\bibitem{mukaimoduli}
{\sc Mukai, S.}
\newblock {\em {An Introduction to Invariants and Moduli}}, vol.~81 of {\em Cambridge Studies in Advanced Mathematics}.
\newblock Cambridge University Press, 2003.

\bibitem{mumfordstability}
{\sc Mumford, D.}
\newblock {Stability of Projective Varieties}.
\newblock {\em L'Enseignement Math{\'e}matique 23\/} (1977).

\bibitem{git}
{\sc Mumford, D., Fogarty, J., and Kirwan, F.}
\newblock {\em {Geometric Invariant Theory}}, vol.~34 of {\em Ergebnisse der Mathematik und ihrer Grenzgebiete (A Series of Modern Surveys in Mathematics)}.
\newblock Springer, 1994.

\bibitem{mfk}
{\sc Mumford, D., Fogarty, J., and Kirwan, F.}
\newblock {\em {Geometric Invariant Theory}}, {Third}~ed.
\newblock Ergebnisse der Mathematik und Ihrer Grenzgebiete. Springer-Verlag, 1994.

\bibitem{newstead}
{\sc Newstead, P.}
\newblock {\em {Introduction to Moduli Problems and Orbit Spaces}}.
\newblock Narosa Publishing House, 2012.

\bibitem{popovvinberg}
{\sc Popov, V., and Vinberg, E.}
\newblock {Invariant Theory}.
\newblock In {\em {Algebraic Geometry IV: Linear Algebraic Groups, Invariant Theory}}, A.~Parshin and I.~Shafarevich, Eds. Springer, 1994, pp.~123--278.

\bibitem{qiaothesis}
{\sc Qiao, Y.}
\newblock {\em {Blow-Ups in Non-Reductive GIT and Moduli Spaces of Unstable Vector Bundles over Curves}}.
\newblock {DPhil thesis}, University of Oxford, 2022.

\bibitem{rossthesis}
{\sc Ross, J.}
\newblock {\em {Instability of Polarised Algebraic Varieties}}.
\newblock PhD thesis, Imperial College, University of London, 2003.

\bibitem{rossthomas}
{\sc Ross, J., and Thomas, R.}
\newblock {A Study of the Hiobert Mumford Criterion for the Stability of Projective Varieties}.
\newblock {\em J. Alg. Geom 16}, 2 (2007), 201--255.

\bibitem{rydhthesis}
{\sc Rydh, D.}
\newblock {\em {Families of Cycles and the Chow Scheme}}.
\newblock PhD thesis, KTH Stockholm, 2008.

\bibitem{schmitt}
{\sc Schmitt, A.}
\newblock {\em {Geometric Invariant Theory and Decorated Principal Bundles}}, vol.~11 of {\em Zürich Lectures in Advanced Mathematics}.
\newblock European Mathematical Society, 2008.

\bibitem{viehweg}
{\sc Viehweg, E.}
\newblock {\em {Quasi-Projective Moduli for Polarized Manifolds}}, vol.~30 of {\em Ergebnisse der Mathematik und ihrer Grenzgebiete}.
\newblock Springer, 1995.

\bibitem{yau}
{\sc Yau, S.-T.}
\newblock {On the Ricci Curvature of a Compact K{\"a}hler Manifold and the Complex Monge--Amp{\'e}re Equation, I}.
\newblock {\em Comm. Pure App. Math. 31}, 3 (1978), 339--411.

\end{thebibliography}

\end{document}